\documentclass[11pt,a4paper]{amsart}
\usepackage{mathmode}
\usepackage{bm}
\usepackage{yhmath}
\usepackage{tikz,tikz-cd}
\usepackage[margin=1.3in]{geometry}
\usepackage{graphicx,array,subfigure}
\usepackage{stackengine,scalerel}
\usepackage{amsthm,enumitem}
\usepackage{multirow}
\usepackage{svg}
\graphicspath{ {./figures/} }
\def\X#1{E_{#1}}   

\def\cups{\cup\cdots\cup}
\def\Z{\mathbb{Z}}

\def\R{\mathbb{R}}

\def\O{\mathcal{O}}
\def\Q{\mathbb{Q}}
\def\R{\mathbb{R}}
\def\slp{\text{slp}^\ast}

\def\slope{\text{slp}}

\def\Aut{\mathrm{Aut}}

\def\Zx{\widehat{\Z}^{\times}}
\def\GL{\mathrm{GL}}
\def\SL{\mathrm{SL}}
\tikzset{%
  symbol/.style={
    draw=none,
    every to/.append style={
      edge node={node [sloped, allow upside down, auto=false]{$#1$}}
    },
  },
}

\def\ker{\mathrm{ker}}
\def\hatsigma{\widehat{\sigma}}

\def\hatpi{\widehat{\pi_1}}         
\def\piorb{\pi_1^{\text{orb}}}      
\def\hatpiorb{\widehat{\piorb}}     
\def\pihat{\hatpi}
\def\hatH{\widehat{H_1}}

\def\lk{\mathrm{lk}}
\def\H{\mathbf{H}}

\title{Profinite rigidity for two-bridge links and 3-tangle Montesinos links}

\author{Tamunonye Cheetham-West}
\address{Yale University, 219 Prospect Street, New Haven, CT, U.S.A, 06511}
\email{tamcheethamwest24@gmail.com}

\author{Xiaoyu Xu}
\address{Beijing International Center for Mathematical Research, Peking University, Beijing, P.R.China, 100871}
\email{xuxiaoyu@stu.pku.edu.cn}

\date{\today}

\begin{document}

\maketitle

\begin{abstract}
For any two-bridge link or 3-tangle Montesinos link  $L\subset S^3$ (including knot), this paper proves that $\pi_1(S^3-L)$ is profinitely rigid among the fundamental groups of compact orientable 3-manifolds.
\end{abstract}

\newsavebox{\pullback}
\sbox\pullback{%
\begin{tikzpicture}%
\draw (0,0) -- (1ex,0ex);%
\draw (1ex,0ex) -- (1ex,1ex);%
\end{tikzpicture}}

\section{Introduction}
For a compact 3-manifold $M$ (possibly with boundary), finite quotients of the fundamental group $\pi_1(M)$ are useful for distinguishing $M$ from other 3-manifolds in practice.  This idea can be well illustrated by its application to   knot complements. For example, when $M$  the exterior of a non-trivial knot $K$ in $S^3$, it is known by \cite{Kup14} that this non-triviality can be certified by finding a non-abelian finite cover of $M$ whose degree depends only on the crossing number of $K$. According to the results of Whitten \cite{Whi74} and Gordon--Luecke \cite{GL89},  two prime knots $K_1$ and $K_2$ in $S^3$ with $\pi_1(S^3-K_1)\cong \pi_1(S^3-K_2)$ are equivalent up to isotopy and mirroring. A further question, stated by Boileau and Friedl \cite[Question 1.3]{BF20}, is whether a prime knot $K$ in $S^3$ is determined, up to isotopy and mirroring, by the set of finite quotients of $\pi_1(S^3-K)$. 

For a finitely generated group $\Gamma$, let $\mathcal{C}(\Gamma)$ denote the set of finite quotient groups of $\Gamma$. Two finitely generated groups $\Gamma$ and $\Delta$ are said to be {\it profinitely equivalent} if $\mathcal{C}(\Gamma)=\mathcal{C}(\Delta)$. This is equivalent to say that $\Gamma$ and $\Delta$ have isomorphic profinite completions $\widehat{\Gamma}\cong \widehat{\Delta}$ \cite{DFPR}. Given a collection $\mathscr{A}$ of finitely generated groups, a group $\Gamma\in \mathscr{A}$ is said to be {\it profinitely rigid in $\mathscr{A}$} if any group $\Delta\in \mathscr{A}$ that is profinitely equivalent to $\Gamma$ is isomorphic with $\Gamma$. 

Throughout this paper, let $\mathscr{M}$ denote the class of fundamental groups of compact orientable 3-manifolds. In this paper, we give a positive answer to a class of examples for Boileau--Friedl's question that also generalizes for links in $S^3$. In particular, we prove:

\begin{theorem}\label{mainthm1}
    For any two-bridge link (including knot) $L\subseteq S^3$, $\pi_1(S^3-L)$ is profinitely rigid in $\mathscr{M}$. 
\end{theorem}


\begin{theorem}\label{mainthm2}
 	For any Montesinos link (including knot) $L\subseteq S^3$ with 3 rational tangles, $\pi_1(S^3-L)$ is profinitely rigid in $\mathscr{M}$. 
\end{theorem}

We remark that two-bridge links are exactly Montesinos links with no more than 2 rational tangles. However, the results are listed separately since they involve slightly different technical details. \autoref{mainthm1} is proven as \autoref{thm: two bridge} and \autoref{mainthm2} is proven as a combination of \autoref{thm: Montesinos 1} and \autoref{thm: Montesinos 2}.

A significance of profinite rigidity is its relation to the isomorphism problem. 
Indeed, if a group $\Gamma$ is profinitely rigid within a class $\mathscr{A}$ of finitely presented groups, then there exists an algorithm that, given a finite presentation of any $\Delta \in \mathscr{A}$, decides in finite time whether $\Gamma$ and $\Delta$ are isomorphic; see \cite[Page 5]{BCR16}. This proceeds by simultaneously searching either for  a finite group belonging to $\mathcal{C}(\Gamma)\vartriangle\mathcal{C}(\Delta)$ -- which terminates in finite time if $\Gamma$ and $\Delta$ are not profinitely equivalent, or for an explicit isomorphism between $\Gamma$ and $\Delta$ -- which terminates in finite time as long as they are isomorphic. 
In particular, the groups $\pi_1(S^3-L)$ as appearing in \autoref{mainthm1} and \autoref{mainthm2} have solvable isomorphism problem in $\mathscr{M}$. This gives a new solution to the recognition problem of two-bridge knots and 3-tangle Montesinos knots, which is different from \cite{Geo93} based on the theory of normal surfaces. 
\begin{corollary}
Fix $K\subseteq S^3$ to be a two-bridge knot or a Montesinos knot with 3 rational tangles. There exists an algorithm that inputs a polygonal knot digram $D$ and determines whether the diagram $D$ represents the knot $K$. 
\end{corollary}
\begin{proof}
Let $K'$ be the knot represented by the digram $D$. One can obtain a finite presentation of $\pi_1(S^3-K')$ via the Wirtinger presentation. Then one apply the preceding algorithm to determine whether $\pi_1(S^3-K')$ is isomorphic to $\pi_1(S^3-K)$. If they are not isomorphic, then clearly $D$  does not represent the knot $K$.  If they are isomorphic, since $K$ is a prime knot, \cite{Whi74} and  \cite{GL89} imply that $D$ represents either $K$ or its mirror image.  One can fix knot diagrams for $K$ and its mirror image, and through a finite search of Reidemeister moves, one can certify that $D$ represents $K$ or that $D$ represents the mirror image of $K$. If $D$ represents $K$, then we are done. Otherwise, if $D$ represents the mirror image of $K$, then $D$ represents $K$ if and only if $K$ is isotopic to its mirror image, and this can be verified through the complete classification for these families of knots (\autoref{thm: two bridge classification}, \autoref{prop: Montesinos classification}).  
\end{proof}
As a historical remark, \autoref{mainthm1} and \autoref{mainthm2} cover some earlier results of knots and links in $S^3$ whose complements have fundamental groups that are profinitely rigid in $\mathscr{M}$. These include the figure-eight knot by Bridson-Reid \cite{BR20} via a distinct approach, and a range of examples by the second author \cite{Xu24} such as the Whitehead link, some families of Pretzel links including $(-2,3,7)$ and $(-2,3,8)$, (full) twist knots, and $\mathbf{b}(10,3)$ etc. More results on profinite rigidity of knot groups can be found at \cite{Wilkes19,BF20,CW23}.

The proofs for \autoref{mainthm1} and \autoref{mainthm2}  are motivated by the observation that the links involved are characterized, up to isotopy and mirror image, by the homeomorphism type of their two-fold branched covers. We briefly outline here the strategy for  \autoref{mainthm1}. Consider as an example that $L\subseteq S^3$ is a hyperbolic two-bridge link, and $N$ is a compact orientable 3-manifold with $\pi_1(N)$ profinitely equivalent with $\pi_1(S^3-L)$. \cite[Theorem 4.1]{Xu24} implies that $N$ is the exterior of a hyperbolic link $L'\subseteq S^3$. The $\pi$-orbifold group of $L$ is a finite dihedral group which appears as a finite quotient of $\pi_1(S^3-L)$. Thus, this finite dihedral group also appears as a finite quotient of $\pi_1(N)$. A careful application of \cite{Xu24} implies that this finite dihedral quotient of $\pi_1(N)$ also arises from the $\pi$-orbifold group of $L'$. Then, a classification theorem (e.g. \cite[Proposition 3.2]{BoileauZimmermann})  can be applied to show that $L'$ is also a two-bridge link. However, there exist non-equivalent two-bridge links with isomorphic $\pi$-orbifold groups. To eliminate this possibility, we apply a theorem of Burde \cite{Bur75} that distinguishes these two-bridge knots via linking number invariants in a finite-sheeted great circle link cover of the link complement corresponding to the kernel of this finite-dihedral quotient.

Although they are not directly applicable to the proofs of \autoref{mainthm1} and \autoref{mainthm2}, we present here two further criteria for the profinite rigidity of knot or link groups in $\mathscr{M}$ that follow the same spirit. 
\begin{theorem}\label{crit1}
Suppose $L$ is a hyperbolic link in $S^3$ which is not a two-bridge link. If the $\pi$-orbifold group of $L$ is profinitely rigid among all $\pi$-orbifold groups of hyperbolic links in $S^3$, then $\pi_1(S^3-L)$ is profinitely rigid in $\mathscr{M}$. 
\end{theorem}

\begin{theorem}\label{crit2}
Let $K$ be a hyperbolic knot in $S^3$. Suppose that for infinitely many integers $r\ge 2$, the $r$-fold cyclic  cover of $S^3$ branched along $K$ has its fundamental group profinitely rigid in $\mathscr{M}$. Then, $\pi_1(S^3-K)$ is also profinitely rigid in $\mathscr{M}$. 
\end{theorem}

\autoref{crit1} and \autoref{crit2} will be proven in Section~\ref{motivatingcriteria}.

\subsection*{Acknowledgments} TCW thanks Alan Reid and Ryan Spitler for helpful conversations about this project. XX would like to thank Yi Liu for discussions.

\section{Preliminaries.}\label{prelim}

\subsection{Profinite completion} In addition to the materials covered in this section, we refer the readers to \cite{RZ10} for a standard reference on profinite groups. 

A {\em profinite group} is an inverse limit of finite groups indexed over a  directed partially ordered set, equipped with the subspace topology inherited from the product
topology. 
The {\em profinite completion} of an abstract  group $\Gamma$ is defined as a profinite group
$$
\widehat{\Gamma}=\varprojlim_{N\in \mathcal N} \Gamma/N
$$ 
where $\mathcal{N}$ is the collection of all finite-index normal subgroups of $\Gamma$ indexed by reverse inclusion. 

There is a canonical homomorphism $i_\Gamma:\Gamma\to\widehat{\Gamma}$ with dense image sending every element $\gamma\in\Gamma$ to a tuple $(\pi_N(\gamma))_{N\in\mathcal{N}}\in \widehat{\Gamma}$ where  $\pi_N:\Gamma\to \Gamma/N$ is the quotient map. 
It is clear that $i_\Gamma$ is injective if and only if $\Gamma$ is residually finite. For ease of notation, when $\Gamma$ is residually finite, we always identify $\Gamma$ as its image in $\widehat{\Gamma}$, and elements in $\Gamma$ 
are also viewed as elements in $\widehat{\Gamma}$.

By a {\em homomorphism of profinite groups} $\Phi:G_1\to G_2$, we always mean a continuous homomorphism. However, a deep theorem of Nikolov--Segal \cite{NS} implies that any abstract  homomorphism from a topologically finitely generated profinite group to a profinite group is continuous. Thus, we shall not stress upon the continuity when $G_1$ and $G_2$ are topologically finitely generated, especially when they are the profinite completions of finitely generated groups. 

The profinite completion encodes the full data of finite quotients. 

\begin{theorem}[\cite{DFPR}]
For two finitely generated abstract groups $\Gamma_1$ and $\Gamma_2$, $\mathcal{C}(\Gamma_1)=\mathcal{C}(\Gamma_2)$ if and only if $\widehat{\Gamma_1}\cong \widehat{\Gamma_2}$. 
\end{theorem}

In addition to the set of finite quotients, the profinite completion of an abstract group also encodes its lattice of finite-index subgroups. 
\begin{proposition}[{\cite[Proposition 3.2.2]{RZ10}}] For an abstract group $\Gamma$, there is an isomorphism between the lattice of finite-index subgroups of $\Gamma$ and the lattice of open subgroups in $\widehat{\Gamma}$ given
as follows.
\begin{equation*}
\begin{tikzcd}[row sep=0.01cm]
\{\text{Finite-index subgroups of }\Gamma\} \arrow[r, "1:1", leftrightarrow] & \{\text{Open subgroups of }\widehat{\Gamma}\}\\
H \arrow[r, maps to] & \overline{i_{\Gamma}(H)}\cong \widehat{H}\\
i_{\Gamma}^{-1}(U) & U \arrow[l,maps to]
\end{tikzcd}
\end{equation*}
In addition, this correspondence sends normal subgroups to normal subgroups.
\end{proposition}
\begin{definition}
Let $\Gamma_1$ and $\Gamma_2$ be abstract groups. Given an isomorphism $\Phi: \widehat{\Gamma_1}\to \widehat{\Gamma_2}$, we say that two finite-index subgroups $H_1\le \Gamma_1$ and $H_2\le \Gamma_2$ are {\em $\Phi$--corresponding} if $\Phi( \overline{i_{\Gamma_1}(H_1)})= \overline{i_{\Gamma_2}(H_2)}$. 
\end{definition}

The profinite completion is also functorial. Any homomorphism between abstract groups $\phi:\Gamma_1\to \Gamma_2$ induces a homomorphism of profinite groups $\widehat{\phi}: \widehat{\Gamma_1}\to \widehat{\Gamma_2}$ in the following sense. Let $\mathcal{N}_1$ and $\mathcal{N}_2$ be the inverse system of finite-index normal subgroups of $\Gamma_1$ and $\Gamma_2$. Then, $\phi^{-1}(\mathcal{N}_2)$ is a subsystem of $\mathcal{N}_1$, and the homomorphism 
\begin{equation*}
\begin{tikzcd}
\widehat{\Gamma_1}\arrow[r] & \prod_{N\in \phi^{-1}(\mathcal{N}_2)} \Gamma_1/N \arrow[r,"\phi"] & \prod_{M\in \mathcal{N}_2}\Gamma_2/M
\end{tikzcd}
\end{equation*}
has its image contained in $\varprojlim_{M\in\mathcal{N}_2}\Gamma_2/M= \widehat{\Gamma_2}$. This resulting map is defined as $\widehat{\phi}$. It is clear that $\hat{\phi}\circ i_{\Gamma_1}=i_{\Gamma_2}\circ\phi$.

\subsection{Profinite completions of 3-manifold groups}
For $M$ a compact 3-manifold, $\pi_1(M)$ is residually finite. This follows by combining a theorem of Hempel \cite{HempelRF} with Agol's positive resolution to Thurston's Virtual Haken Conjecture \cite{AgolHaken}. As such, 3-manifolds have an abundance of finite covers. Since we are only focusing on the fundamental groups, we always assume that compact 3-manifolds stated in this paper do not have boundary spheres. 

In their influential papers, Wilton--Zalesskii showed that profinite completions of 3-manifold groups detect geometric decompositions. The following theorem serves as a brief conclusion of their results, see also \cite{WilkesJSJ} for generalizations to 3-manifolds with boundary. 
\begin{theorem}[\cite{WZ17,WZ17b,WZ2}]\label{thm: Wilton-Zalesskii}
For a compact orientable 3-manifold $M$ with empty or toral boundary, the isomorphism type of $\hatpi(M)$ determines whether $M$ is geometric in the sense of Thurston. It determines the geometry type when $M$ is geometric, and determines the prime decomposition and the JSJ-decomposition when $M$ is non-geometric.
\end{theorem}

Profinite rigidity for seven of the eight Thurston geometries is well understood, see \cite{Fun13} for $Sol$ geometry and \cite{Hem14,Wil17} for Seifert fibered spaces. The question of profinite rigidity for hyperbolic 3-manifold groups remains.

\subsection{Cohomology of profinite groups}
The standard reference for cohomology theory of profinite groups is \cite[Chapter 6]{RZ10}.  A complete review of profinite cohomology theory is beyond the scope of this article, so we will only present a brief introduction here. 

For a profinite group $G$, the profinite cohomology of $G$ -- denoted by $\H^\ast(G;-)$ -- is the continuous cohomology of $G$ usually defined for coefficients in discrete $G$-modules, and the profinite homology of $G$ -- denoted by $\H_\ast(G;-)$ -- is the continuous homology of $G$ usually defined for coefficients in profinite $G$-modules, where the modules are equipped with continuous $G$-actions. One may take the following properties as equivalent definitions for profinite (co)homology:
\begin{enumerate}[label=(\arabic*),leftmargin=*]
\item for any discrete $G$-module $A$, $\H^k(G;A)=\varinjlim_U H^k(G/U;\mathrm{Fix}_U A)$,
\item for any profinite $G$-module $B$, $\H_k(G;B)=\varprojlim_U H_k(G/U; \mathrm{Cofix}_U B)$,
\end{enumerate}
where $U$ ranges through all open normal subgroups in $G$ in both statements, see \cite[Corollary 6.5.6 and Corollary 6.5.8]{RZ10}.

Let $\Phi: G_1\to G_2$ be a homomorphism of profinite groups. Then, $\Phi$ induces homomorphisms $\Phi^\ast: \H^\ast(G_2;A)\to \H^\ast(G_1;A)$ for every discrete $G_2$-module $A$, and $\Phi_\ast: \H_\ast(G_1;B)\to \H_\ast(G_2;B)$ for every profinite $G_2$-module $B$.

Given a homomorphism $\phi: \Gamma\to G$ from an abstract group $\Gamma$ to a profinite group $G$, $\phi$ also induces a  homomorphism  $\phi^\ast: \H^\ast(G;A)\to H^\ast(\Gamma;A)$ for every discrete $G$-module $A$. We also note that any finite discrete $\Gamma$-module is naturally a $\widehat{\Gamma}$-module. 
\begin{definition}
An abstract group $\Gamma$ is {\em cohomologically good} if for every finite $\Gamma$-module $M$, the homomorphism $i_\Gamma^\ast: \H^\ast(\widehat{\Gamma};M)\to H^\ast(\Gamma;M)$ induced by the canonical homomorphism is an isomorphism. 
\end{definition}

The concept of cohomological goodness was first introduced by Serre in \cite{cohomologyGalois}. In this paper, we mainly apply this concept to virtually surface groups.
\begin{proposition}[{\cite[Lemma 3.2 and Proposition 3.6]{GJZ}}]\label{prop: goodness}
Virtually surface groups are cohomologically good. 
\end{proposition}

\section{Dehn filling}

\subsection{Peripheral regularity}
\begin{definition}
Suppose $M$ and $N$ are compact orientable 3-manifolds with toral boundary, and $\Phi: \hatpi(M)\to \hatpi(N)$ is an isomorphism. We say that $\Phi$ is {\em peripheral regular} if there exists a homeomorphism $h: \partial M\to \partial N$, for which we denote as $h_i: \partial_iM \to \partial_iN$ on each component, with the following properties. For each $i$, there exists $g_i\in \hatpi(N)$ such that the following diagram commutes.
\begin{equation*}
\begin{tikzcd}
\hatpi(\partial_iM) \arrow[rr, "\widehat{{h_i}_\ast}"] \arrow[d,"\widehat{\text{incl}_{\ast}}"]& &\hatpi(\partial_iN)\arrow[d,"\widehat{\text{incl}_{\ast}}"]\\
\hatpi(M) \arrow[r,"\Phi"] &\hatpi(N) \arrow[r,"C_{g_i}"] &\hatpi(N)
\end{tikzcd}
\end{equation*}
Here $C_{g_i}(x)=g_ixg_i^{-1}$ denotes the conjugation by $g_i$ on the left. 
To specify the homeomorphism, we also say that $\Phi$ is peripheral regular with respect to $h$.
\end{definition}

\begin{theorem}
\label{thm: peripheral regular}
Suppose $M$ and $N$ are oriented cusped hyperbolic 3-manifolds. Then, any isomorphism $\Phi: \hatpi(M)\to \hatpi(N)$ is peripheral regular with respect to some $h_\Phi: \partial M \xrightarrow{\cong} \partial N$.  The homeomorphism $h_{\Phi}$ is unique up to isotopy, and $h_{\Phi}$ is either orientation preserving on all components or orientation reversing on all components.
\end{theorem} 
\begin{proof}
    The peripheral regularity of $\Phi$ follows from \cite[Theorem 1.4]{Xu25}. The uniqueness of $h_{\Phi}$ follows directly from the malnormality of the closures of the peripheral subgroups $\{\overline{\pi_1(\partial_1N)},\cdots, \overline{\pi_1(\partial_iN)}\}$ in $\hatpi(N)$ \cite[Lemma 4.5]{WZ17}. And the fact that $h_\Phi$ is orientation-preserving on all components or orientation reversing on all components follows from \cite[Proposition 7.1]{Xu25}.
\end{proof}

We say that the isomorphism $\Phi:\hatpi(M)\to \hatpi(N)$ is \textit{orientation preserving} (resp. \textit{orientation reversing}) if the homeomorphism $h_\Phi:\partial M\to \partial N$ is orientation preserving (resp. orientation reversing). In fact, this is equivalent to the alternative that the profinite mapping degree $\mathbf{deg}(\Phi)$ equals $1$ or $-1$, see \cite[Proposition 7.1]{Xu25}.  

\subsection{Aligning Dehn fillings}
Let $M$ be a compact  3-manifold with non-empty boundary consisting of tori $\partial_1M, \cdots, \partial_n M$. The {\em slopes} on $\partial_i M$ is defined as 
$$
\slope(\partial_iM)= \{\text{free homotopy classes of  unoriented essential simple closed curves on }\partial_iM \}.
$$
Let $\slp(\partial_iM)=\slope(\partial_iM)\cup\{\varnothing\}$, where the symbol $\varnothing$ denotes an `empty slope', and let $\slp(\partial M)= \slp(\partial_1M)\times \cdots \times \slp(\partial_nM)$. 

\def\c{\mathbf{c}}

\begin{definition}
Let $M$ be a compact   3-manifold with non-empty boundary consisting of $n$ tori. For $\c=(c_1,\cdots, c_n)\in \slp(\partial M)$, {\em the Dehn filling of $M$ along $\c$}, denoted by $M_{\c}$, is a compact 3-manifold constructed from $M$ by gluing solid tori to the boundary components $\partial_iM$ where $c_i\neq \varnothing$, so that the meridians of the solid tori are attached to the slopes $c_i\in \slope(\partial_iM)$. 
\end{definition}
\begin{theorem}[{\cite[Theorem A]{Xu24}}]\label{thm: dehn filling}
Let $M$ and $N$ be compact orientable 3-manifolds with toral boundary, and let $\Phi: \hatpi(M)\to \hatpi(N)$ be an isomorphism which is peripheral regular with respect to $h:\partial M\to \partial N$. Then, for any choice of slopes $\mathbf{c}$ on $\partial M$, there is an isomorphism $\Psi: \hatpi (M_{\mathbf{c}})\to \hatpi  (N_{h(\mathbf{c})})$ that fits into the following commutative diagram.
\begin{equation*}
\begin{tikzcd}
\hatpi(M)\arrow[r, "\Phi"]\arrow[d,"\widehat{\text{incl}_{\ast}}"] & \hatpi(N) \arrow[d,"\widehat{\text{incl}_{\ast}}"]\\
\hatpi (M_{\mathbf{c}}) \arrow[r,"\Psi"] &  \hatpi  (N_{h(\mathbf{c})})
\end{tikzcd}
\end{equation*} 
\end{theorem}

\section{Link complements}
Fix an orientation of the ambient space $S^3$  throughout this paper. 
By a link $L$ in $S^3$, we always mean a tame link. We denote by $n(L)$ the open regular neighborhood of $L$, and by $N(L)$ the closure of $n(L)$. Finally,  $\X{L}=S^3\setminus n(L)$ denotes  the  compact exterior of $L$. 

A link $L$ is {\em component-ordered} if we fix an ordering of its components as $L=K_1\cups K_n$. For a component-ordered link $L=K_1\cups K_n$, the \textit{linking matrix} $\lk(L)$ is an $n\times n$ integral matrix with entries $\lk(L)_{i,j}=\lk(K_i,K_j)$ when $i\neq j$, and $\lk(L)_{i,i}=0$.

Suppose $L=K_1\cups K_n$ is an oriented link  in $S^3$. 
For each $1\le i\le n$, let $(m_i,l_i)$ be the meridian-longitude basis of $\pi_1(\partial N(K_i))$ such that $l_i$ follows the orientation of $K_i$, and $(m_i,l_i)$ is positively oriented with respect to the boundary orientation of $N(K_i)$ inherited from $S^3$. We also identify $(m_i,l_i)$ with their images in $\pi_1(\X{L})$ through the homomorphism $\pi_1(\partial N(K_i))\to \pi_1(\X{L})$ up to a choice of basepoints, and regard them as the {\em preferred meridian-longitude basis} for $L$. Then,  $m_i$ and $l_i$ generate a subgroup isomorphic to $\Z$ or $\Z^2$ in $\pi_1(\X{L})$ depending on whether $K_i$ is a splitted unknotted component. 

\subsection{Perfect isomorphism}

\begin{definition}
Let $L=K_1\cups K_n$ and $L'=K_1'\cups K_n'$ be two component-ordered  oriented  links  in $S^3$, and let $(m_i,l_i)$ and $(m_i',l_i')$ be the preferred meridian-longitude bases for $L$ and $L'$.  An isomorphism  $\Phi: \hatpi(\X L)\to \hatpi(\X{L'})$  is called {\em a perfect isomorphism} if $\Phi$ is peripheral regular with respect to a homeomorphism $h:\partial (\X L)\to \partial(\X{L'})$ such that $h(m_i)=m_i'$ and $h(l_i)=l_i'$.
\end{definition}

\begin{proposition}\label{prop: perfect linking number}
Suppose $L=K_1\cups K_n$ and $L'=K_1'\cups K_n'$ are two component-ordered oriented   links  in $S^3$,  and  $\Phi: \hatpi(\X L)\to \hatpi(\X{L'})$  is a perfect isomorphism. Then $\lk(L)=\lk(L')$. 
\end{proposition}
\begin{sloppypar}
\begin{proof}
For a fixed index $1\le j \le n$, we show that $\lk(K_i,K_j)=\lk(K_i',K_j')$ for any $i\neq j$. 
Note that $H_1(\X{L})$ is the free $\Z$-module over the basis $\{[m_1],\cdots, [m_n]\}$, and by our choice of orientation, $[l_j]=\sum_{i\neq j}\lk(K_i,K_j)[m_i]$ in $H_1(\X{L})$. Similarly, $H_1(\X{L'})$ is the free $\Z$-module over the basis $\{[m_1'],\cdots, [m_n']\}$, and $[l_j']=\sum_{i\neq j}\lk(K_i',K_j')[m_i]$. The profinite abelianization $\hatpi(\X L)^{\mathrm{Ab}}$ is isomorphic to $\hatH (\X{L}) $, which is the free $\widehat{\Z}$-module over the basis $\{[m_1],\cdots, [m_n]\}$, and similarly, $\hatpi(\X {L'})^{\mathrm{Ab}}\cong \hatH (\X{L'}) $ is the free $\widehat{\Z}$-module over the basis $\{[m_1'],\cdots, [m_n']\}$. 

By definition, $\Phi$ sends each $m_i$ or $l_i$ to a conjugate of $m_i'$ or $l_i'$. Thus, the profinite abelianization $\Phi^{\mathrm{Ab}}: \hatH (\X{L}) \to \hatH (\X{L'}) $ sends each $[m_i]$ or $[l_i]$ to $[m_i']$ or $[l_i']$. In particular, $\sum_{i\neq j}\lk(K_i,K_j)[m_i']=\Phi^{\mathrm{Ab}}(\sum_{i\neq j}\lk(K_i,K_j)[m_i])= \Phi^{\mathrm{Ab}}([l_j])=[l_j']=\sum_{i\neq j}\lk(K_i',K_j')[m_i']$ in the free $\widehat{\Z}$-module $\hatH (\X{L'}) $. Thus, $\lk(K_i,K_j)=\lk(K_i',K_j')$ as profinite integers, and hence $\lk(K_i,K_j)=\lk(K_i',K_j')$ as integers for all $i\neq j$.  
\end{proof}
\end{sloppypar}

A link $L$ in $S^3$ is called a {\em hyperbolic link} if $S^3\setminus N(L)$ is a finite-volume hyperbolic 3-manifold. 

\begin{lemma}\label{lem: criteria for perfect}
Suppose $L=K_1\cups K_n$ and $L'=K_1'\cups K_n'$ are two component-orderd oriented hyperbolic links  in $S^3$, and suppose $\Phi: \hatpi(\X{L})\to \hatpi(\X{L'})$ is an isomorphism, which is peripheral regular with respect to $h_{\Phi}:\partial (\X{L})\to \partial (\X{L'})$ according to \autoref{thm: peripheral regular}.  Then $\Phi$ is a  perfect isomorphism if $\Phi$ is orientation preserving and $h_\Phi(m_i)=m_i'$ for each $1\le i \le n$.
\end{lemma}
\begin{proof}
By definition, it suffices to show that $h_{\Phi}(l_j)=l_j'$ for each $1\le j \le n$. 
Similar to the proof of \autoref{prop: perfect linking number}, on the abelianization level we have  $\Phi^{\mathrm{Ab}}([l_j])=[h_{\Phi}(l_j)]$ and $\Phi^{\mathrm{Ab}}([m_i])=[h_{\Phi}(m_i)]=[m_i']$. Thus, $[h_{\Phi}(l_j)]=\Phi^{\mathrm{Ab}}([l_j])=\Phi^{\mathrm{Ab}} (\sum_{i\neq j} \lk(K_i,K_j) [m_i]) = \sum_{i\neq j}\lk(K_i,K_j)[m_i']$ in $\hatH (\X{L'}) $. Since  $\hatH (\X{L'}) $ is  an  extension of scalar of the $\Z$-module $H_1(\X{L'})$ via $\Z\hookrightarrow \widehat{\Z}$, we indeed have  $[h_{\Phi}(l_j)]= \sum_{i\neq j}\lk(K_i,K_j)[m_i']$ in $H_1(\X{L'})$. 
In particular, $h_{\Phi}(l_j)$, viewed as a loop in $\X{L'}$, has zero linking number with the knot $K_j'$. Since, $h_{\Phi}(l_j)$ represents a primitive essential simple closed curve on $\partial N(K_j)$, the only possibilities are $h_\Phi(l_j)=\pm l_j$. Finally, $h_{\Phi}(l_j)=l_j'$ since $h_{\Phi}$ is orientation-preserving.  
\end{proof}

The next theorem elaborates on  \cite[Theorem 4.1]{Xu24}.

\begin{theorem}\label{thm: hyperbolic link}
Let $L$ be an $n$-component oriented hyperbolic link in $S^3$. Suppose $N$ is a compact orientable 3-manifold, and  $\Phi: \hatpi(\X{L})\to \hatpi(N)$ is an isomorphism. Then there exists an $n$-component oriented hyperbolic link in $S^3$ such that $N\cong \X{L'}$; and through this homeomorphism, $\Phi: \hatpi(\X{L})\to \hatpi(\X{L'})$ is a perfect isomorphism. 
\end{theorem}

\begin{proof}
By \cite[Lemma A.1]{Xu24} and \autoref{thm: Wilton-Zalesskii}, $N$ is cusped hyperbolic. 
According to \autoref{thm: peripheral regular}, we suppose that $\Phi$ is peripheral regular with respect to $h_{\Phi}$. 
Let $\X{L}$ inherit its orientation from $S^3$, and fix an orientation on $N$ so that $\Phi$ is orientation-preserving. 
The Dehn filling along the meridians $\mathbf{m}=(m_1,\cdots, m_n)$ of $\X{L}$ yields $S^3$. 
According to \autoref{thm: dehn filling},  
$\hatpi(N_{h_{\Phi}(\mathbf{m})})\cong \hatpi(M_{\mathbf{m}}) \cong \hatpi(S^3)$ is the trivial group. Hence, $\pi_1(N_{h_{\Phi}(\mathbf{m})})$ is trivial by its residual finiteness, and due to the validity of the Poincar\'e conjecture, ${N_{h_{\Phi}(\mathbf{m})}}\cong S^3$, for which we equip the orientation coherent with that on $N$. The core curves of the Dehn filled solid tori in $N_{h_{\Phi}(\mathbf{m})}$ form an $n$-component link $L'$ in $S^3$ such that $N\cong \X{L'}$, and each $h_\Phi(m_i)$ represents a meridian of the $i$-th component of $L'$. $L'$ is a hyperbolic link since $N$ is cusped hyperbolic, and we can choose orientations on $L'$ so that each $h_\Phi(m_i)$ is the preferred meridian. Then, $\Phi: \hatpi(\X{L})\to \hatpi(\X{L'})$ is a perfect isomorphism according to \autoref{lem: criteria for perfect}.
\end{proof}

\subsection{Branched coverings}
Let $L=K_1\cups K_n$ be a link in $S^3$. Let $\sigma:\pi_1(\X{L})\to H_1(\X{L}) \to \Z/2\Z$ be the unique homomorphism that sends each $m_i$ to the element $1\pmod 2$. The {\em balanced two-fold cover} of $\X{L}$, denoted by $C_2(L)$, is the two-fold cover given by $\ker(\sigma)$. Note that each boundary subgroup of $\pi_1(\X{L})$ surjects $\Z/2\Z$ via $\sigma$, so $C_2(L)$ also has $n$ boundary components.   Each meridian loop $m_i$ in $\X{L}$ lifts to a connected two-fold cover in $C_2(L)$, defining  a boundary slope on the $i$-th boundary component of $C_2(L)$ denoted by $2m_i$. The Dehn filling of $C_2(L)$ along the boundary slopes $(2m_1,\cdots, 2m_n)$ is a closed manifold denoted by $M_2(L)$. In fact, $M_2(L)$ is the {\em two-fold branched cover of $S^3$ along $L$}. Throughout this paper, $C_2(L)$ and $M_2(L)$ are equipped with the orientation lifted from $\X{L}$, which is inherited from the oriented ambient space $S^3$. 

We note that the definition of $C_2(L)$ and $M_2(L)$ do not depend on an orientation on $L$. However, in the following proposition, orientations are added to apply  the definition of a perfect isomorphism.

\begin{proposition}\label{prop: branched cover}
    Suppose $L=K_1\cups K_n$ and $L'=K_1'\cups K_n'$ are two component-ordered oriented $n$-component links  in $S^3$,  and  $\Phi: \hatpi(\X{L})\to \hatpi(\X{L'})$  is a perfect isomorphism. Then, there exist isomorphisms $\check{\Phi}:\hatpi(C_2(L))\to \hatpi(C_2(L'))$ and $\phi: \hatpi(M_2(L))\to \hatpi(M_2(L'))$ that fit into the following commutative diagram.
    \begin{equation}\label{diag: branched cover}
        \begin{tikzcd}[column sep=large]
            \hatpi(\X{L}) \arrow[d,"\Phi"',"\cong"] & \hatpi(C_2(L)) \arrow[l,hook'] \arrow[d,"\check{\Phi}"',"\cong"] \arrow[r,two heads,"\widehat{\text{incl}_\ast}"] & \hatpi(M_2(L)) \arrow[d,"\phi"',"\cong"]\\
            \hatpi(\X{L'})   & \hatpi(C_2(L')) \arrow[l, hook'] \arrow[r, two heads, "\widehat{\text{incl}_\ast}"] & \hatpi(M_2(L'))
        \end{tikzcd}
    \end{equation}
\end{proposition}
\begin{sloppypar}
\begin{proof}
    First, it is easy to verify the following commutative diagram.
    \begin{equation*}
        \begin{tikzcd}
            \widehat{\sigma}:\; \hatpi(\X{L}) \arrow[r,two heads] \arrow[d,"\Phi"'] & {\hatpi(\X{L})^{\mathrm{Ab}}\cong \widehat{\Z}[m_1]\oplus \cdots \oplus \widehat{\Z}[m_n]} \arrow[d,"\Phi^{\mathrm{Ab}}"',"{[m_i]\mapsto [m_i']}"] \arrow[rr,"{[m_i]\mapsto 1}"] & &  \Z/2\Z \arrow[d,"\cong"]\\
            \widehat{\sigma'}:\; \hatpi(\X{L'}) \arrow[r,two heads] &  {\hatpi(\X{L'})^{\mathrm{Ab}}\cong \widehat{\Z}[m_1']\oplus \cdots \oplus \widehat{\Z}[m_n']} \arrow[rr,"{[m_i']\mapsto 1}"] &&  \Z/2\Z
        \end{tikzcd}
    \end{equation*}
    Thus, $\Phi$ sends $\ker(\widehat{\sigma})=\hatpi(C_2(L))$ to $\ker(\widehat{\sigma'})=\hatpi(C_2(L'))$. This gives the left block of (\ref{diag: branched cover}) with $\check{\Phi}$ being the restriction of $\Phi$.

    In addition, since $\Phi$ is peripheral regular (with respect to $h_{\Phi}$), $\check{\Phi}$ --   being the restriction of $\Phi$ -- is also peripheral regular,   with respect to $\check h:\partial C_2(L)\to \partial C_2(L')$ that fits into the following commutative diagram. 
    \begin{equation*}
        \begin{tikzcd}
            \partial C_2(L) \arrow[r,"\check h"] \arrow[d,"\text{cover}"'] & \partial C_2(L')\arrow[d,"\text{cover}"]\\
            \partial (\X{L}) \arrow[r, "h_\Phi"] & \partial (\X{L'})
        \end{tikzcd}
    \end{equation*}Hence,  $\check h (2m_i)=2m_i'$. Recall that $M_2(L)$ is the Dehn filling of $C_2(L)$ along $(2m_1,\cdots,2m_n)$ and $M_2(L')$ is the Dehn filling of $C_2(L')$ along $(2m_1',\cdots,2m_n')$. Thus, \autoref{thm: dehn filling} gives an isomorphism $\phi: \hatpi(M_2(L))\to \hatpi(M_2(L'))$ that fits into the right block of (\ref{diag: branched cover}). 
\end{proof}
\end{sloppypar}

\section{Two-bridge links}

\subsection{Schubert normal form}
Two-bridge links (including knots) were introduced by Schubert \cite{Sch56} and classified by their Schubert normal forms $\mathbf{b}(\alpha,\beta)$, where $-\alpha<\beta<\alpha$ are coprime integers and $\beta$ is odd.  In particular, $\mathbf{b}(\alpha,\beta)$ is a knot when $\alpha$ is odd, and  $\mathbf{b}(\alpha,\beta)$ is a two-component link when $\alpha$ is even. Conventionally, $\alpha$ is called the {\em torsion} and $\beta$ is called the {\em crossing number} of $\mathbf{b}(\alpha,\beta)$. 

The Schubert diagram of a two-bridge link consists of two arcs above the projection plane, depicted as straight segments, and two arcs below the projection plane. The Schubert  diagram is always equipped with a standard orientation, such that the two arcs above the projection plane are assigned opposite orientations, see \autoref{fig: two bridge}. 

When we mention a two-bridge link $L=\mathbf{b}(\alpha,\beta)$, $L$ is always equipped with this orientation. 

\begin{figure}[ht!]
\centering
\includegraphics[width=5cm]{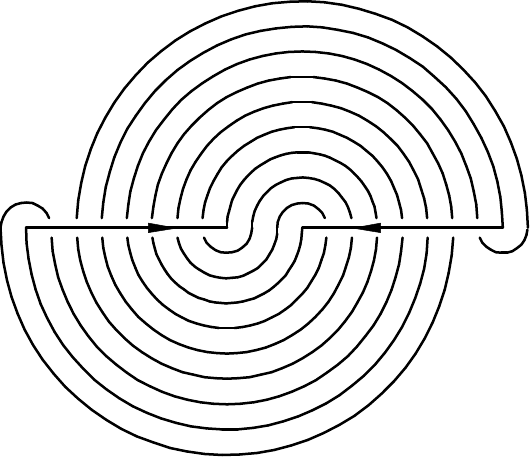}
\caption{The oriented Schubert diagram for $\mathbf{b}(8,3)$ (the Whitehead link)}
\label{fig: two bridge}
\end{figure}

\begin{theorem}[\cite{Sch56}]\label{thm: two bridge classification}
Let $L=\mathbf{b}(\alpha,\beta)$ and $L=\mathbf{b}(\alpha',\beta')$ be two-bridge links. 
\begin{enumerate}[leftmargin=*,label=(\arabic*)]
\item $L$ and $L'$ are isotopic as oriented links if and only if $\alpha=\alpha'$ and $\beta'\equiv \beta^{\pm1}\pmod{2\alpha}$.
\item $L$ and $L'$ are isotopic as unoriented links if and only if $\alpha=\alpha'$ and $\beta'\equiv \beta^{\pm1}\pmod\alpha$.
\end{enumerate}
\end{theorem}

\begin{remark}\label{rmk: orientation two bridge}
    All two-bridge links are invertible; and for two-component two-bridge links, the two components are interchangeable, and altering the orientation of one component switches $\mathbf{b}(\alpha,\beta)$ to $\mathbf{b}(\alpha,\beta\pm \alpha)$. Thus, given an unoriented two-bridge link, any assignment of orientation is in fact a standard orientation within some Schubert normal form. 
\end{remark}

\subsection{Double branched cover}
For a two-bridge link (including knot) $L=\mathbf{b}(\alpha,\beta)$, it is known that $M_2(L)$ is the (oriented) lens space $\mathrm{L}(\alpha,\beta)$ (see \cite[Proposition 12.3]{BZ}  for example). Conversely, Hodgson and Rubinstein \cite{HR85} proved, through classifying involutions on lens spaces,  that  two-bridge links are actually characterized by their two-fold  branched covers. 
\begin{theorem}[\cite{HR85}]\label{thm: involution lens space}
    Suppose $L$ is an unoriented  link in $S^3$ so that the two-fold branched cover $M_2(L)$ is the oriented lens space  $\mathrm{L}(\alpha,\beta)$.  Then, $L$ is isotopic to the two-bridge link $\mathbf{b}(\alpha,\beta^{\ast})$, where $\beta^\ast\equiv \beta \pmod \alpha$. 
\end{theorem}

\subsection{Burde's classification theorem}\label{sec: Burde}
Although non-isotopic (unoriented) two-bridge links are distinguished by the   homeomorphism type of their two-fold branched covers $M_2(L)$, they may not be distinguished by the isomorphism type of  $\pi_1(M_2(L))$ since non-homeomorphic lens spaces could possibly have isomorphic fundamental groups. A classification theorem of Burde \cite{Bur75},   based on linking number invariants in a further cover, is precisely applicable to distinguish these cases.

For a two-bridge link  $L=\mathbf{b}(\alpha,\beta)$,   the universal covering space of $M_2(L)= \mathrm{L}(\alpha,\beta)$ is  homeomorphic to $S^3$, which we denote as $\widetilde{S^3}$ in order to distinguish from the original space $S^3$. Denote $\check{L}=r^{-1}(L)\subseteq \mathrm{L}(\alpha,\beta)$ and $\widetilde{L}= q^{-1}(L)\subseteq \widetilde{S^3}$, where $r$ and $q$ are the (branched) covering maps defined as follows.
\begin{equation*}
\begin{tikzcd}[column sep=4cm]
        q: \,(\widetilde{S^3},\widetilde{L}) \arrow[r,"p","\text{$\alpha$-fold cover}"'] & (\mathrm{L}(\alpha,\beta),\check{L}) \arrow[r,"r","\text{two-fold branched cover}"'] & (S^3,L)
\end{tikzcd}
\end{equation*}

 We equip $\widetilde{S^3}$ with the orientation lifted from $S^3$, and equip $\widetilde{L}$ with the orientation lifted from $L$. 

When $\alpha$ is odd, both $L$ and $\check{L}$ have one component, $\check{L}$ is null-homotopic in $\mathrm{L}(\alpha,\beta)$, and $\widetilde{L}=\widetilde{L}_1\cups \widetilde{L}_\alpha$ has $\alpha$ components. The deck transformation group $\Z/\alpha\Z$ for $p:\widetilde{S^3} \to \mathrm{L}(\alpha,\beta)$ acts transitively on the $\alpha$ components of $\widetilde{L}$. Let $\tau$ be a generator of $\Z/\alpha\Z$. We say  that  the  ordering of the components $\widetilde{L}_1,\cdots,\widetilde{L}_\alpha$  is {\em $\tau$-labeled} if $\tau(\widetilde{L}_i)=\widetilde{L}_{i+1}$ with indices mod $\alpha$.

When $\alpha$ is even, $L=J\cup K$ is a two-component link. $\check{L}=\check{J}\cup \check{K}$ also has two components, and both of them represent the unique $2$-torsion in $\pi_1(\mathrm{L}(\alpha,\beta))\cong \Z/\alpha\Z$. Consequently, $\widetilde{L}=\widetilde{J}_1\cups \widetilde{J}_{\frac{\alpha}{2}}\cup \widetilde{K}_1\cups \widetilde{K}_{\frac{\alpha}{2}}$ has $\alpha$ components, where $\widetilde{J}_i$ projects $J$ and $\widetilde{K}_i$ projects $K$ through $q$. The deck transformation group $\Z/\alpha\Z$ for $p$ acts on $\widetilde{L}$ with two orbits $\{\widetilde{J}_1,\cdots,\widetilde{J}_\frac{\alpha}{2}\}$ and $\{\widetilde{K}_1,\cdots,\widetilde{K}_\frac{\alpha}{2}\}$, so that the order-$2$ subgroup $\frac{\alpha}{2}\Z/\alpha\Z$ preserves all components of $\widetilde{L}$. Let $\tau$ be a generator of $\Z/\alpha\Z$. We say that  the ordering of the components $\widetilde{J}_1,\cdots,\widetilde{J}_\frac{\alpha}{2},\widetilde{K}_1,\cdots,\widetilde{K}_\frac{\alpha}{2} $ is {\em $\tau$-labeled} if $\tau(\widetilde{J}_i)=\widetilde{J}_{i+1}$ and $\tau(\widetilde{K}_i)=\widetilde{K}_{i+1}  $ 
        with indices mod $\frac{\alpha}{2}$.

\begin{theorem}[\cite{Bur75}]\label{thm: Burde}
    Let $L=\mathbf{b}(\alpha,\beta)$ and $L'=\mathbf{b}(\alpha,\beta')$ be two two-bridge links equipped with the standard orientation. Suppose $\tau$ and $\tau'$ are generators of the deck transformation groups of $p: \widetilde{S^3} \to \mathrm{L}(\alpha,\beta)$ and $p':\widetilde{S^3} \to \mathrm{L}(\alpha,\beta')$ isomorphic to $\Z/\alpha\Z$. 
    Suppose $\widetilde{L}$ and $\widetilde{L}'$ are component-ordered so that $\widetilde{L}$ is $\tau$-labeled and $\widetilde{L}'$  is $\tau'$-labeled. If $\lk(\widetilde  L)=\lk(\widetilde L')$, then $L$ and $L'$ are isotopic as oriented links in $S^3$. 
\end{theorem}

\begin{remark}
    In fact, $\widetilde{L}$ and $\widetilde{L'}$ are great circle links in $\widetilde{S^3}$, so all the linking numbers belong to $\{\pm 1\}$. Thus, it is important to follow the precise orientations lifted from $L\subseteq S^3$ and $L'\subseteq S^3$ to obtain the correct signs. 
\end{remark}

\subsection{Proof of profinite rigidity}

\begin{theorem}\label{thm: two bridge}
    Let $L\subseteq S^3$ be a two bridge link (including knot). Then $\pi_1(\X{L})$ is profinitely rigid in $\mathscr{M}$.
\end{theorem}
\begin{proof}
    According to \cite[Corollary 2]{Men84}, a two-bridge link  is either a torus link or a hyperbolic link. In the former case, $\X{L}$ is a Seifert fibered space, and it follows from   \cite[Corollary 8.3 and Remark 8.4]{Xu25A} that $\pi_1(\X{L})$ is profinitely rigid in $\mathscr{M}$. In the following, we assume that $L=\mathbf{b}(\alpha,\beta)$ is a hyperbolic two-bridge link. 

    Suppose $N$ is a compact orientable 3-manifold with $\hatpi(N)\cong \hatpi(\X L)$.  We show that $N$ is homeomorphic to $\X L$.
By \autoref{thm: hyperbolic link}, there is an oriented hyperbolic link $L'$ such that $N\cong \X{L'}$, and there is a perfect isomorphism $\Phi:\hatpi(\X{L})\to \hatpi(\X{L'})$.  
Note that the two-fold branched cover $M_2(L)$ is the lens space $\mathrm{L}(\alpha,\beta)$. According to \autoref{prop: branched cover}, $\hatpi(M_2(L'))\cong\hatpi(M_2(L))\cong \Z/\alpha\Z$. Therefore, $\pi_1(M_2(L'))\cong \Z/\alpha\Z$ since it is residually finite, so $M_2(L')$ is also a lens space with the same parameter $\alpha$. 
 \autoref{thm: involution lens space} then implies that $L'$, ignoring the orientation, is a two-bridge link with the same torsion $\alpha$. According to \autoref{rmk: orientation two bridge}, the orientation on $L'$ can be realized as the standard orientation of a Schubert normal  form $\mathbf{b}(\alpha,\beta')$. 

\def\-{\setminus}
\def\L{\mathrm{L}}


\begin{sloppypar}
In order to show $\mathbf{b}(\alpha,\beta')\cong \mathbf{b}(\alpha,\beta)$, we apply Burde's classification theorem. 
Construct the branched covers $(\widetilde{S^3},\widetilde{L})\to (\L(\alpha,\beta),\check L)\to (S^3,L)$ and $(\widetilde{S^3},\widetilde{L'})\to (\L(\alpha,\beta'),\check{L'}) \to (S^3, L')$ as in Subsection~\ref{sec: Burde}, and now we display the relation between them under $\Phi$.  
Recall that $\L(\alpha,\beta)\-n(\check{L})$ is the balanced two-fold cover $C_2(L)$ and $\L(\alpha,\beta)$ is the two-fold branched cover $M_2(L)$. By \autoref{prop: branched cover},  we obtain a commutative diagram:
\begin{equation*}
    \begin{tikzcd}[column sep=large]
            \hatpi(\X L) \arrow[d,"\Phi"',"\cong"] & \hatpi(\L(\alpha,\beta)\-n(\check{L})) \arrow[l,hook'] \arrow[d,"\check{\Phi}"',"\cong"] \arrow[r,two heads,"\rho=\widehat{\text{incl}_\ast}"] & \hatpi(\L(\alpha,\beta))\cong \Z/\alpha\Z \arrow[d,"\phi"',"\cong"]\\
            \hatpi(\X{L'})   & \hatpi(\L(\alpha,\beta')\-n(\check{L'})) \arrow[l, hook'] \arrow[r, two heads, "\rho'=\widehat{\text{incl}_\ast}"] & \hatpi(\L(\alpha,\beta'))\cong \Z/\alpha\Z
        \end{tikzcd}
\end{equation*}
Note that $\widetilde{S^3}\-n(\widetilde{L})$ is the $\alpha$-fold cover of $\L(\alpha,\beta)\-n(\check{L})$ corresponding to $\ker(\rho)$, and so is $\widetilde{S^3}\-n(\widetilde{L'})$ corresponding to $\ker(\rho')$. Thus, they are $\check{\Phi}$-corresponding finite covers, and we have the following commutative diagram
\begin{equation}\label{diag: alpha cyclic}
        \begin{tikzcd}
            1 \arrow[r]  &  \hatpi(\widetilde{S^3}\-n(\widetilde{L})) \arrow[r] \arrow[d,"\widetilde{\Phi}"',"\cong"] & \hatpi(\L(\alpha,\beta)\-n(\check{L})) \arrow[r,"\rho"] \arrow[d,"\check{\Phi}"',"\cong "] & \Z/\alpha\Z \arrow[r] \arrow[d,"\cong","\phi"'] & 1\\
             1 \arrow[r]  &  \hatpi(\widetilde{S^3}\-n(\widetilde{L'})) \arrow[r]   & \hatpi(\L(\alpha,\beta')\-n(\check{L'})) \arrow[r,"\rho'"]   & \Z/\alpha\Z \arrow[r]   & 1
        \end{tikzcd}
\end{equation}
where $\widetilde{\Phi}$ is the restriction of $\check{\Phi}$.

\end{sloppypar}

Let us now order the components of $\widetilde{L}$ and $\widetilde{L'}$. 
Fix a generator $\tau\in \Z/\alpha\Z=\hatpi(\L(\alpha,\beta))$, and let $\widetilde{L}=K_1\cups K_{\alpha}$ be $\tau$-labeled. According to \cite[Lemma 6.3]{Xu25A} and \cite[Lemma 4.5]{WZ17}, the isomorphism $\widetilde{\Phi}$ sends the conjugacy classes of the closures of the peripheral subgroups of $\hatpi(\widetilde{S^3}\-n(\widetilde{L}))$ bijectively to the conjugacy classes of the closures of the peripheral subgroups of $\hatpi(\widetilde{S^3}\-n(\widetilde{L'}))$. Order the components of $\widetilde{L'}$ as  $\widetilde{L'}=K_1'\cups K_{\alpha}'$ so that $\widetilde{\Phi}$ sends the conjugacy class of   $\overline{\pi_1(\partial N(K_i))}$ to the conjugacy class of $\overline{\pi_1(\partial N(K_i'))}$ for each $1\le i\le \alpha$. Group theoretically,  $\Z/\alpha \Z$ acts via outer automorphisms on the conjugacy classes of the closures of the peripheral subgroups of $\hatpi(\widetilde{S^3}\-n(\widetilde{L}))$  and $\hatpi(\widetilde{S^3}\-n(\widetilde{L'}))$, which is consistent with the deck transformation  group acting on the boundary components. 
The commutative diagram (\ref{diag: alpha cyclic}) then implies that the $\widetilde{\Phi}$-bijection between the conjugacy classes of the closures of  the   peripheral subgroups is $\Z/\alpha\Z$-equivariant via $\phi:\Z/\alpha\Z\to \Z/\alpha\Z$. Thus, with $\tau'=\phi(\tau)\in \Z/\alpha\Z$ being the corresponding generator, $\widetilde{L'}=K_1'\cups K_{\alpha}'$  is $\tau'$-labeled.

Finally, we show that $\widetilde{\Phi}$ is a perfect isomorphism. In fact, $\Phi$ is peripheral regular with respect to the homeomorphism $h_{\Phi}: \partial(\X{L})\to \partial (\X{L'})$ that matches up the preferred meridian-longitude bases. Thus, $\widetilde{\Phi}$ -- being the restriction of $\Phi$ -- is also peripheral regular with respect to $\widetilde{h}:\partial(\widetilde{S^3}\- n(\widetilde{L})) \to\partial(\widetilde{S^3}\- n(\widetilde{L'})) $  that fits into the following commutative diagram. 
\begin{equation}\label{cov boundary}
    \begin{tikzcd}
        \partial(\widetilde{S^3}\- n(\widetilde{L}) )\arrow[r,"\widetilde{h}"] \arrow[d,"\text{cover}"'] & \partial(\widetilde{S^3}\- n(\widetilde{L'})  )\arrow[d,"\text{cover}"]\\  
        \partial(\X{L}) \arrow[r,"h_\Phi"] & \partial (\X{L'})
    \end{tikzcd}
\end{equation}
By construction, $\widetilde{h}(\partial N(K_i))= \partial N(K_i')$ for each $1\le i \le \alpha $.   
The vertical covering maps in (\ref{cov boundary}) are $\pi_1$-injective on each component, and send the preferred meridians in $\widetilde{S^3}\-n(\widetilde{L})$ and $\widetilde{S^3}\-n(\widetilde{L'})$  
to   the  twice of the preferred meridians in  $\X{L}$ and $\X{L'}$. 
Thus, $\widetilde{h}$ sends the preferred meridian of $K_i$ to the preferred meridian of $K_i'$. 
In addition,  $\widetilde{S^3}\-n(\widetilde{L})$  and $\widetilde{S^3}\-n(\widetilde{L'})$  are equipped with the orientations lifted from $\X{L}$ and $\X{L'}$. 
Since $h_\Phi$ is orientation-preserving, the commutative diagram (\ref{cov boundary}) implies that $\widetilde{h}$  is also orientation-preserving. Hence, \autoref{lem: criteria for perfect} implies that $\widetilde{\Phi}$ is a perfect isomorphism. 

Then, \autoref{prop: perfect linking number} implies that   the two component-ordered oriented links $\widetilde{L}$ and $\widetilde{L'}$ in $\widetilde{S^3}$ have the same linking matrix: $\lk(\widetilde{L})=\lk(\widetilde{L'})$. According to \autoref{thm: Burde}, $L$ and $L'$ are isotopic (as oriented links), and therefore, $N\cong \X{L'}$ is homeomorphic to $\X{L}$. 
\end{proof}

\section{Montesinos links}
\subsection{Definition and classification} 
\begin{sloppypar}
Montesinos links are generalizations of two-bridge links. A Montesinos link is constructed by combining $n$ rational tangles of slopes $r_1=\frac{p_1}{q_1},\cdots, r_n=\frac{p_n}{q_n} \in \mathbb{Q}$ left-to-right in order, and then taking a numerator closure as shown in \autoref{fig: Montesinos}.  A Montesinos link so described is denoted by $\mathbf{M}(\frac{p_1}{q_1},\cdots, \frac{p_n}{q_n})$, where $p_i$ and $q_i$ are coprime integers. In this paper, we require that: 
\begin{enumerate}[leftmargin=*,label=(\roman*)]
\item $r_i\neq \infty$, since then the Montesinos link can be decomposed as a connected sum of two-bridge links;
\item $r_i\neq 0$, since the rational tangles of slope $0$ can be omitted;
\item $r_i\notin \Z$ if $n\ge 2$, since a rational tangle with integral slope can be merged into an adjacent tangle to obtain new rational tangle that gives a simpler presentation of the Montesinos link. 
\end{enumerate}

\begin{figure}[ht!]
\centering
\subfigure[$\mathbf{M}(\frac{p_1}{q_1},\cdots, \frac{p_n}{q_n})$]{\includegraphics[width=6.7cm]{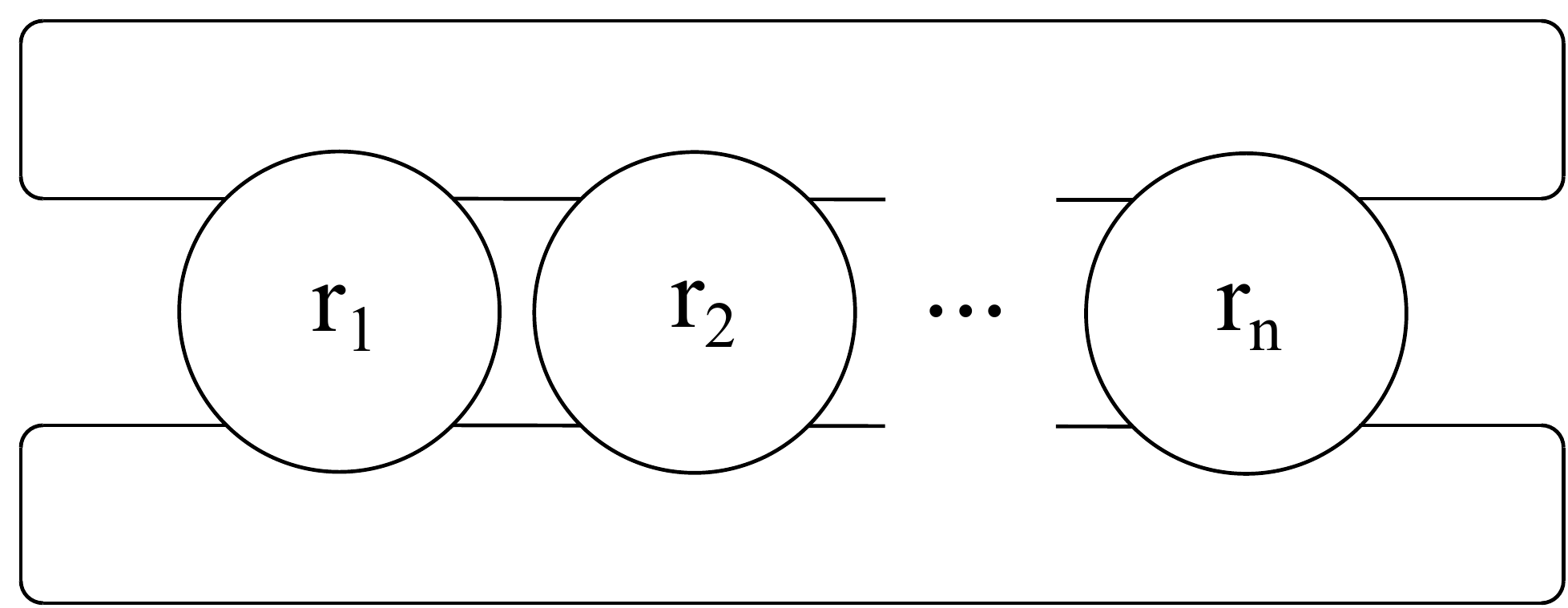}} 
\hspace{0.5cm}
\subfigure[$\mathbf{M}(\frac{3}{2},-\frac{2}{3}, \frac{1}{4})$]{\includegraphics[width=6.7cm]{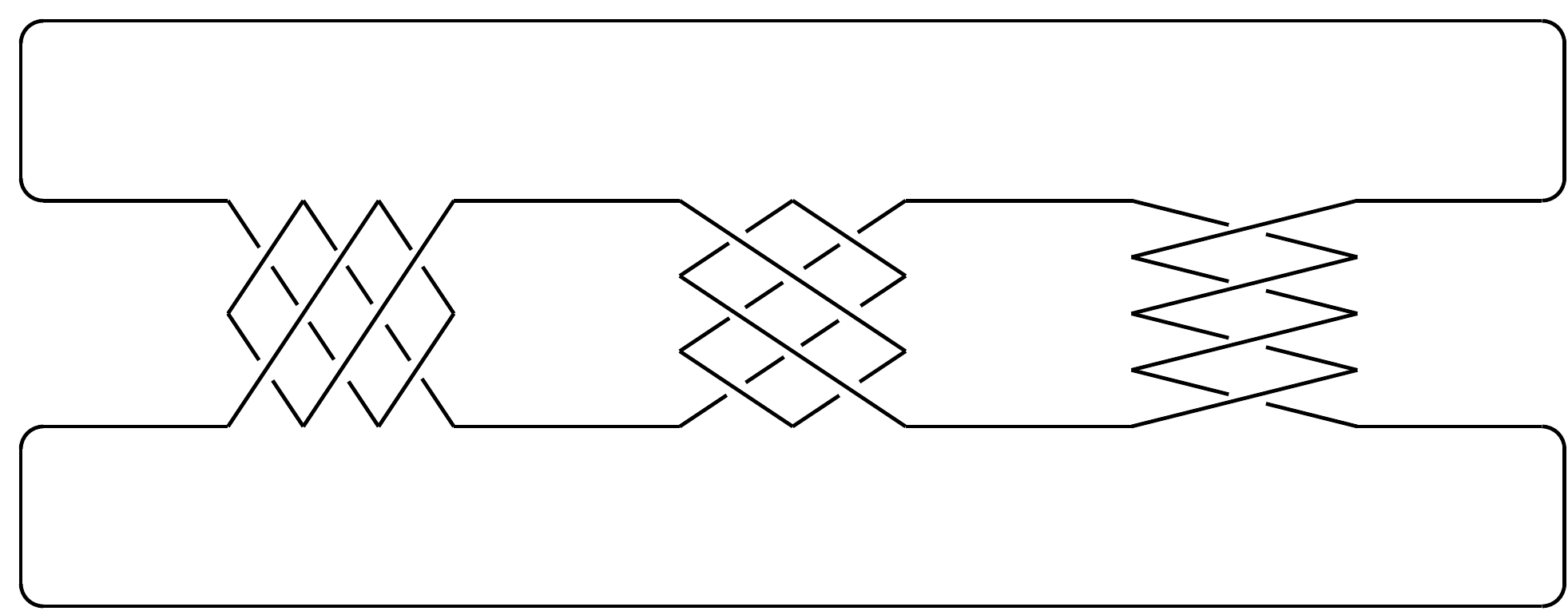}} 
\caption{Montesinos links}
\label{fig: Montesinos}
\end{figure}

We remind the readers  that   Montesinos links with no more than $2$ tangles are exactly the  two-bridge links.

\begin{proposition}\label{prop: Montesinos classification}
    \begin{enumerate}[label=(\arabic*), leftmargin=*]       
 \item The mirror image of a Montesinos link $\mathbf{M}(\frac{p_1}{q_1},\cdots, \frac{p_n}{q_n})$ is $\mathbf{M}(-\frac{p_1}{q_1},\cdots, -\frac{p_n}{q_n})$.
        \item Suppose $L=\mathbf{M}(\frac{p_1}{q_1},\cdots, \frac{p_n}{q_n})$ and $L'=\mathbf{M}(\frac{p_m'}{q_m'},\cdots, \frac{p_m'}{q_m'})$ are two unoriented Montesinos links with $m,n\ge 2$. Then, $L$ and $L'$ are isotopic if and only if $m=n$, $\sum_{i=1}^{n} \frac{p_i}{q_i}=\sum_{i=1}^{m}\frac{p_i'}{q_i'}$, and  up to cyclic permutation and reversal of order, $\frac{p_i}{q_i}\equiv \frac{p_i'}{q_i'}\pmod 1$ for each $i$. 
    \end{enumerate}
\end{proposition}
\begin{proof}
The first statement is clear from definition; and the second statement was proven by \cite{Bonahon}, see also \cite{Zie84} or \cite{BS16}. 
\end{proof}

In the following proposition, a {\em Seifert link} is a link $L$ in $S^3$ such that $\X{L}$ is a Seifert fibered space. 


\begin{proposition}[{\cite{Oer84} and \cite[Theorem A.8]{BS16}}]\label{prop: Montesinos geometrization classification}
Any Montesinos link $L$ falls into exactly one of the following three possibilities. 
\begin{enumerate}[label=(\arabic*), leftmargin=*]
    \item $L$ is a hyperbolic link.
    \item\label{M.classification.2} $L$ is  a Seifert link  when it is isotopic to mirror image to one of: $\mathbf{b}(n,1)$ ($n > 0$), $\mathbf{M}(-\frac{1}{2},\frac{1}{2},\frac{1}{p})$ ($p\neq 0$), $\mathbf{M}(-\frac{1}{2},\frac{1}{3},\frac{1}{3})$, $\mathbf{M}(-\frac{1}{2},\frac{1}{3},\frac{1}{4})$, $\mathbf{M}(-\frac{1}{2},\frac{1}{3},\frac{1}{5})$.
    \item\label{M.classification.3} $L$ is isotopic or mirror image to one of the following exceptional cases: $\mathbf{M}(\frac{2}{3},-\frac{1}{3},-\frac{1}{3})$, $\mathbf{M}(\frac{1}{2},-\frac{1}{4},-\frac{1}{4})$, $\mathbf{M}(\frac{1}{2},-\frac{1}{3},-\frac{1}{6})$, $\mathbf{M}(\frac{1}{2},\frac{1}{2},-\frac{1}{2},-\frac{1}{2})$,  where $\X{L}$ is a graph manifold.
\end{enumerate}
\end{proposition}

\end{sloppypar}

\subsection{Double branched cover}
\begin{proposition}[\cite{Mon73}]\label{prop: Montesinos two fold cover}
    Let $L$ be a Montesinos link $\mathbf{M}(\frac{q_1}{p_1},\cdots,\frac{q_n}{p_n})$. Then, the double branched cover $M_2(L)$ is the (oriented) Seifert fibered space $(0,0;\frac{q_1}{p_1},\cdots,\frac{q_n}{p_n})$.
\end{proposition}

\begin{corollary}\label{prop: Montesinos characterize}
    Suppose $L$ and $L'$ are unoriented Montesinos links in $S^3$, and the double branched covers $M_2(L)$ and $M_2(L')$ are homeomorphic. If $L$ has exactly $3$ rational tangles, then $L'$ is isotopic or mirror image to $L$. 
\end{corollary}
\begin{proof}
 Suppose $L= \mathbf{M}(\frac{q_1}{p_1},\frac{q_2}{p_2},\frac{q_3}{p_3})$ and $L'=\mathbf{M}(\frac{p_m'}{q_m'},\cdots, \frac{p_m'}{q_m'})$. According to \autoref{prop: Montesinos two fold cover},  $M_2(L)\cong (0,0;\frac{q_1}{p_1},\frac{q_2}{p_2},\frac{q_3}{p_3})$ is homeomorphic to $M_2(L')\cong (0,0;\frac{p_m'}{q_m'},\cdots, \frac{p_m'}{q_m'})$. According to the classification of Seifert fibered spaces \cite[Theorem VI.17]{Jaco},  these two present the same Seifert fibration. That is to say, $m=3$, and there exist  $\epsilon\in\{\pm1\}$ and a bijection $\sigma:\{1,2,3\}\to \{1,2,3\}$ such that $\frac{p_1}{q_1}+\frac{p_2}{q_2}+\frac{p_3}{q_3}=\epsilon(\frac{p_1'}{q_1'}+\frac{p_2'}{q_2'}+\frac{p_3'}{q_3'})$ and $\frac{p_i}{q_i}\equiv \epsilon\frac{p_{\sigma(i)}'}{q_{\sigma(i)}'}\pmod 1$. Then, regardless of the choice of $\epsilon$ and $\sigma$,  \autoref{prop: Montesinos classification} implies that  $L'$ is isotopic or mirror image to $L$.  
\end{proof}

\begin{remark}
When $L$ has at least $4$ tangles, the Montesinos link $L'$ with $M_2(L)\cong M_2(L')$ could differ from $L$ by reordering its rational tangles. As such, $L'$ differs from $L$ by a sequence of mutations. 
\end{remark}

\begin{proposition}\label{prop: two fold Seifert}
    Suppose $L$ is a link in $S^3$ such that the two-fold branched cover $M_2(L)$ admits a Seifert fibration over the $2$-sphere with $3$ exceptional fibers. Then, $L$ is either a Seifert link or a Montesinos link.  
\end{proposition}
\begin{proof}
    When $|\pi_1(M_2(L))|<\infty$, the result follows from a complete  classification  of spherical orbifolds with underlying space $S^3$ by \cite{Dun88} (see also \cite[Theorem 1.2]{MS25} for another proof via classifying involutions on spherical manifolds with 1-dimensional fixed points). And when $|\pi_1(M_2(L))|=\infty$, the conclusion follows from finding an involution-invariant Seifert fibration on $M_2(L)$; an explicit proof can be found in \cite[Proposition 3.3]{Mot17}. 
\end{proof}

\subsection{Distinguishing Seifert fibered spaces}

In this subsection, we prove a variation of Wilkes' result \cite{Wil17} for profinite rigidity among Seifert fibered spaces. 

\def\G{\Gamma}

\begin{definition}
Let $\G_1$ and $\G_2$ be finitely generated groups. An isomorphism $\Phi:\widehat{\G_1}\to \widehat{\G_2}$ is called {\em regular} if the abelianization of $\Phi$,  $\Phi^{\mathrm{Ab}}: \widehat{\G_1}^{\mathrm{Ab}}\to \widehat{\G_2}^{\mathrm{Ab}}$, is the profinite completion of an isomorphism $\G_1^{\mathrm{ab}}\to \G_2^{\mathrm{ab}}$. The isomorphism     $\Phi:\widehat{\G_1}\to \widehat{\G_2}$  is called {\em strongly regular} if for any pair of $\Phi$--corresponding finite-index subgroups $\G_1'\le \G_1$ and $\G_2'\le \G_2$,  the restricted isomorphism $\Phi': \widehat{\G_1'}\to \widehat{\G_2'}$  is regular. 
\end{definition}

\begin{proposition}\label{prop: Seifert strongly regular}
    Let $M$ and $N$ be closed orientable Seifert fibered spaces which are not lens spaces, and suppose $\Phi: \hatpi(M)\to \hatpi(N)$ is a strongly regular isomorphism. Then, $M$ and $N$ are homeomorphic. 
\end{proposition}

Before proving \autoref{prop: Seifert strongly regular}, we need a couple of lemmas. The next lemma is a generalization of \cite[Sub-lemma 1]{Xu25}.

\begin{lemma}\label{lem:strong regular criteria}
    Let $G$, $H$, $K$, $L$ be finitely generated groups, and let $p : G \to K$ and $q : H \to L$
be group homomorphisms such that $p(G)$ has finite index in $K$ and $q(H)$ has finite index in
$L$. Suppose $\Phi : \widehat G \to  \widehat H$ and $\Psi : \widehat K \to  \widehat L$ are isomorphisms such that the following diagram commutes. 
\begin{equation*}
\begin{tikzcd}
        \widehat{G} \arrow[d,"\widehat{p}"'] \arrow[r,"\Phi"] & \widehat{H} \arrow[d,"\widehat{q}"]\\
        \widehat{K} \arrow[r,"\Psi"]& \widehat{L}
\end{tikzcd}
\end{equation*}
If $\Phi$ is strongly regular, then $\Psi$ is also strongly regular. 
\end{lemma}
\begin{proof}
   Let  $K'\le K$ and $L'\le L$ be any $\Psi$--corresponding finite-index subgroups, and let $\Psi':\widehat{K'}\to \widehat{L'}$ be the restriction of $\Psi$. We show that $\Psi'$ is regular.  
   
   Let $G'=p^{-1}(K')$ and $H'=q^{-1}(L')$, which are finite-index subgroups in $G$ and $H$, and let $p':G'\to K'$ and $q':H'\to L'$ be the restrictions of $p$ and $q$. Then, $p'(G')=p(G)\cap K'$ has finite index in $K'$, and $q'(H')=q(H)\cap L$ has finite index in $L'$. The commutative diagram implies that  $G'$ and $H'$ are $\Phi$-corresponding finite-index subgroups of $G$ and $H$.  Let $\Phi':\widehat{G'}\to \widehat{H'}$ be the restriction of $\Phi$. Then, $\widehat{q'}\circ \Phi'= \Psi'\circ \widehat{p'}$. By assumption, $\Phi'$ is regular, so \cite[Sub-lemma 1]{Xu25} implies that $\Psi'$ is also regular, which finishes the proof. 
\end{proof}

\begin{lemma}\label{lem: non-trivial image}
    Let $\O$ be a closed  orientable 2-orbifold with non-positive orbifold Euler characteristic, and let $\Sigma$ be a finite-sheeted cover of $\O$ so that $\Sigma$ is a surface. Then, the map $H^2(\piorb(\O);\Z)\to H^2(\pi_1(\Sigma);\Z)\cong \Z$ induced by the covering map has non-trivial image.
\end{lemma}
\begin{proof}
    Let $M$ be a Seifert fibered space over $\O$ with non-zero Euler number. Then, we obtain a central extension $1\to \Z\to \pi_1(M)\to \piorb(\O)\to 1$. Pulling this Seifert fibration back along the orbifold covering $\Sigma\to \O$ yields a finite sheeted horizontal covering $\widetilde{M}\to M$, so that $\widetilde{M}$ is a Seifert fibered space over $\Sigma$ with Euler number $e(\widetilde{M})=[\Sigma:\O]e(M)\neq 0$. Consequently, we obtain a pull-back diagram of central extensions 
    \begin{equation*}
        \begin{tikzcd}
            1 \arrow[r] & \Z \arrow[r] \arrow[d, equal] & \pi_1(\widetilde{M}) \arrow[dr, phantom, "\usebox\pullback" , very near start, color=black] \arrow[r] \arrow[d] & \pi_1(\Sigma) \arrow[r] \arrow[d,"\text{cover}_\ast"] & 1\\
            1 \arrow[r] & \Z \arrow[r] & \pi_1(M) \arrow[r] & \piorb(\O) \arrow[r] & 1
        \end{tikzcd}
    \end{equation*}
    so that the first row does not split. Thus, the covering map sends the cohomology class in $H^2(\piorb(\O);\Z)$ corresponding to the second row to a non-trivial cohomology class in $H^2(\pi_1\Sigma;\Z)$ corresponding to the first row. 
\end{proof}

\begin{lemma}\label{lem: surface mapping degree}
    Let $\Sigma$ be a closed orientable surface  of genus $g\ge 1$. Suppose $\phi:\hatpi(\Sigma)\to \hatpi(\Sigma)$ is a regular isomorphism. Then, for any positive integer $n$, $\phi^\ast: \Z/n\cong \H^2(\hatpi(\Sigma);\Z/n)\to \H^2(\hatpi(\Sigma);\Z/n)\cong \Z/n$ is a scalar multiplication by $\pm 1$. 
\end{lemma}
\begin{proof}
    Suppose that the abelianization $\phi_\ast: \hatH (\Sigma;\Z) \to \hatH(\Sigma;\Z) $ is the profinite completion of an isomorphism $h: H_1(\Sigma;\Z)\to H_1(\Sigma;\Z)$. Note that  
    $\Sigma$ admits a degree $1$ map $F$ to the torus $T^2$.  Since $\pi_1(T^2)\cong \Z^2$ is abelian, the homomorphism between the fundamental groups factors through the abelianization as $F_\ast: \pi_1(\Sigma)\to H_1(\Sigma;\Z)\xrightarrow{f}\pi_1(T^2)\cong \Z^2$. We define another homomorphism $G:\pi_1(\Sigma)\to H_1(\Sigma;\Z)\xrightarrow{h}H_1(\Sigma;\Z)\xrightarrow{f}\pi_1(T^2)$. By construction, the following diagram commutes. 
    \begin{equation*}
        \begin{tikzcd}
\hatpi(\Sigma) \arrow[dd, "\widehat{G}"'] \arrow[rrr, "\phi"] \arrow[rd] &                                                                             &                                                     & \hatpi(\Sigma) \arrow[dd, "\widehat{F_\ast}"] \arrow[ld] \\
                                                                         & \hatH (\Sigma;\Z)  \arrow[r, "\widehat{h}"] & \hatH (\Sigma;\Z)  \arrow[rd, "\widehat{f}"'] &                                                          \\
\hatpi(T^2) \arrow[rrr, "id"]                                            &                                                                             &                                                     & \hatpi(T^2)                                             
\end{tikzcd}
    \end{equation*}
Thus, on the homology level, we have a commutative diagram (see \cite[Proposition 5.19]{Xu25}).
\def\tensor{\widehat{\Z}\otimes_{\Z}}
\begin{equation*}
 \begin{tikzcd}
     \tensor H_2(\pi_1(\Sigma);\Z) \cong \H_2(\hatpi(\Sigma);\widehat{\Z}) \arrow[r,"\phi_\ast"] \arrow[d,"\widehat{G}_\ast"] & \H_2(\hatpi(\Sigma);\widehat{\Z}) \cong \tensor H_2(\pi_1(\Sigma);\Z) \arrow[d,"\widehat{F}_\ast"] \\ 
     \tensor H_2(\pi_1(T^2);\Z) \cong \H_2(\hatpi(T^2);\widehat{\Z}) \arrow[r,"id"]   & \H_2(\hatpi(T^2);\widehat{\Z}) \cong \tensor H_2(\pi_1(T^2);\Z)
 \end{tikzcd}   
\end{equation*}
Since $F$ is a degree $1$ map, $\widehat{F}_\ast$ sends the fundamental class $1\otimes[\Sigma]$ to $1\otimes [T^2]$; and since $\widehat{G}$ is the profinite completion of the abstract group homomorphism  $G:\pi_1(\Sigma)\to \pi_1(T^2)$, the goodness \autoref{prop: goodness} implies that $\widehat{G}_\ast$ sends   $1\otimes[\Sigma]$ to $m\otimes [T^2]$ for some $m\in \Z$. Moreover, $\phi$ is an isomorphism, so $\phi_\ast$ sends $1\otimes [\Sigma]$ to $\lambda\otimes [\Sigma]$ for some $\lambda\in\Zx$. The commutative diagram then implies $\lambda=m$, so $\lambda=m=\pm1$. 
Finally, by the universal coefficient theorem, for any $n\in \mathbb{N}$, $\phi^\ast: \Z/n\cong \H^2(\hatpi(\Sigma);\Z/n)\to \H^2(\hatpi(\Sigma);\Z/n)\cong \Z/n$ is also  a scalar multiplication by $\pm 1$. 
\end{proof}

\begin{lemma}\label{lem: orbifold mapping degree}
    Let $\O_1$ and $\O_2$ be two closed orientable 2-orbifolds with non-positive orbifold Euler characteristics. Suppose $\psi: \hatpiorb(\O_1)\to \hatpiorb(\O_2)$ is a strongly regular isomorphism. Then, $\O_1$ and $\O_2$ are isomorphic; and there exists an isomorphism $\O_1\cong \O_2$ so that under this identification, $\psi^\ast:\H^2(\O_1;\Z/n)\to \H^2(\O_2;\Z/n)$ is a scalar multiplication by $\pm 1$ for any positive integer $n$. 
\end{lemma}
\begin{sloppypar}
\begin{proof}
    First of all, given any isomorphism $\psi: \hatpiorb(\O_1)\to \hatpiorb(\O_2)$, \cite[Proposition 5.7]{Wil17} shows that $\O_1$ and $\O_2$ are isomorphic, and there exists a unit $\kappa\in \Zx$ and an identification $\O_1\cong \O_2\cong \O$ such that for any positive integer $n$, $\psi^{\ast}$ acting on $\H^2(\O;\Z/n)$ is a scalar multiplication by $\kappa$. Thus, it suffices to show that $\kappa=\pm 1$. 

    Suppose, by contrary, that $\kappa\neq \pm 1$. Then, there exists some positive integer $m$ such that $\kappa \not\equiv\pm1\pmod{m}$. Let $\Sigma\to \O$ be a finite-sheeted characteristic cover so that $\Sigma$ is a surface; for instance, we can first find a finite surface cover $\Sigma_0\to \O$, and then let $\pi_1(\Sigma)$ be the intersection of all the subgroups in $\piorb(O)$ with index no more than $[\Sigma_0:\O]$. With this construction, $\psi(\overline{\pi_1(\Sigma)})=\overline{\pi_1(\Sigma)}$, and we let $\phi: \hatpi(\Sigma)\to \hatpi(\Sigma)$ be the restriction of $\psi$. By assumption, $\phi$ is regular.     

    Let $[\alpha]\in H^2(\pi_1(\Sigma);\Z)\cong \Z$ be a generator, and let $i:\pi_1(\Sigma)\to \piorb(\O)$ be the inclusion map. According to \autoref{lem: non-trivial image}, the homomorphism $i^\ast : H^2(\piorb(\O);\Z)\to H^2(\pi_1(\Sigma);\Z)$ has non-trivial image, so  we assume that for some $[\omega]\in H^2(\piorb(\O);\Z)$,  $i^\ast[\omega]=l[\alpha]$ for some positive integer $l$. Let $n=ml$, and let $[\alpha]_n$ be the image of $[\alpha]$ in $H^2(\pi_1(\Sigma);\Z/n)\cong \Z/n$, which is again a generator. Similarly, let $[\omega]_n$ be the image of $[\omega]$ in $H^2(\piorb(\O);\Z/n)$. Then,  $i^\ast([\omega]_n)=l[\alpha]_n$.  

    On the cohomology level, by \autoref{prop: goodness}, we have the following commutative diagram:
    \begin{equation*}
        \begin{tikzcd}
            H^2(\pi_1(\Sigma);\Z/n)\cong \H^2(\hatpi(\Sigma);\Z/n ) & \H^2(\hatpiorb(\O);\Z/n)\cong H^2(\piorb(\O);\Z/n) \arrow[l,"\widehat{i}^\ast"']\\
            H^2(\pi_1(\Sigma);\Z/n)\cong \H^2(\hatpi(\Sigma);\Z/n ) \arrow[u,"\phi^\ast"]& \H^2(\hatpiorb(\O);\Z/n)\cong H^2(\piorb(\O);\Z/n) \arrow[l,"\widehat{i}^\ast"'] \arrow[u,"\psi^\ast"]
        \end{tikzcd}
    \end{equation*}
    Then, $\widehat{i}^\ast \circ \psi^\ast$ sends $[\omega]_n\in H^2(\piorb(\O);\Z/n) \cong \H^2(\hatpiorb(\O);\Z/n) $ to $\widehat{i}^\ast(\kappa [\omega]_n)=\kappa l [\alpha]_n\in H^2(\pi_1(\Sigma);\Z/n)\cong \H^2(\hatpi(\Sigma);\Z/n )$. On the other hand, $\phi$ is regular, so according to \autoref{lem: surface mapping degree}, $\phi^\ast$ is a scalar multiplication by $\pm1$. Thus, $\phi^\ast\circ \widehat{i}^\ast$ sends $[\omega]_n$ to $\phi^\ast(l[\alpha]_n)=\pm l[\alpha]_n$. Recall that $[\alpha]_n$ is a generator of $H^2(\pi_1(\Sigma);\Z/n)\cong \Z/n$, so we have $\kappa l\equiv \pm l \pmod{n}$. Since $n=ml$, we derive that $\kappa\equiv \pm 1\pmod{m}$, which contradicts with our assumption. 
\end{proof}
\end{sloppypar}

\begin{proof}[Proof of \autoref{prop: Seifert strongly regular}]
    According to \cite{Wil17}, we only need to consider the case when $M$ and $N$ admit $\mathbb{H}^2\times \mathbb{E}^1$ geometry. In this case, $M$ and $N$ admit unique Seifert fibrations, and we denote their base orbifolds as $\O_M$ and $\O_N$. By \cite{Wil17}, $\O_M$ and $\O_N$ have the same orientability. When $\O_M$ and $\O_N$ are both non-orientable, let $\widetilde{M}$ and $\widetilde{N}$ be the two-fold covers of $M$ and $N$ which pull back the orientable two-fold covers of $\O_M$ and $\O_N$. Then, $M\cong N$ if and only if $\widetilde{M}\cong \widetilde{N}$; and $\pi_1(\widetilde{M})$ and $\pi_1(\widetilde{N})$ are $\Phi$-corresponding finite-index subgroups of $\pi_1(M)$ and $\pi_1(N)$ since their closures are the centralizers of the unique maximal virtually central normal procyclic subgroup in $\hatpi(M)$ and $\hatpi(N)$ (see \cite[Theorem 5.2]{Wil17}), so $\Phi|:\hatpi(\widetilde{M})\to \hatpi(\widetilde{N})$ is also strongly regular. Therefore, we are justified to assume that both $\O_M$ and $\O_N$ are orientable. 

    In this case, by \cite[Theorem 5.2]{Wil17}, there is a commutative diagram of isomorphisms between the central extensions:
    \begin{equation*}
            \begin{tikzcd}
        1 \arrow[r] & \widehat{\Z} \arrow[r] \arrow[d,"\phi"',"\cong"] & \hatpi(M) \arrow[r] \arrow[d,"\Phi"',"\cong"] & \hatpiorb(\O_M) \arrow[r]  \arrow[d,"\psi"',"\cong"] & 1\\
         1 \arrow[r] & \widehat{\Z} \arrow[r] & \hatpi(N) \arrow[r] & \hatpiorb(\O_N) \arrow[r] & 1
    \end{tikzcd}
    \end{equation*}
    On one hand, since $M$ and $N$ admit $\mathbb{H}^2\times \mathbb{E}^1$ geometry and their base orbifolds are orientable, their regular fibers represent non-torsion homology classes in $H_1(M;\Z)$ and $H_1(N;\Z)$. Since $\Phi$ is regular, the isomorphism $\phi: \widehat{\Z}\to \widehat{\Z}$ between the closures of the fiber subgroups must be a scalar multiplication by $\pm 1$, i.e. $\phi$ sends the regular fiber of $M$ to a $\pm 1$-multiple of the regular fiber of $N$. On the other hand, according to \autoref{lem:strong regular criteria}, the assumption that $\Phi $ is strongly regular implies that $\psi: \hatpiorb(\O_M)\to \hatpiorb(\O_N)$ is also strongly regular. Then, \autoref{lem: orbifold mapping degree} implies that under an identification $\O_M\cong \O_N\cong \O$, $\psi^\ast$ acting on $\H^2(\hatpiorb(\O);\Z/n)$ is a scalar multiplication by $\pm 1$ for any positive integer $n$. Finally, by Proposition 5.7 and Theorem 5.1 of \cite{Wil17}, $M$ and $N$ must be homeomorphic as they form a Hempel pair with scale factor $\pm1$.  
\end{proof}
\subsection{The hyperbolic case}
\begin{theorem}\label{thm: Montesinos 1}
    Let $L$ be a hyperbolic Montesinos link with 3 rational tangles. Then, $\pi_1(\X L)$ is profinitely rigid in $\mathscr{M}$. 
\end{theorem}
\begin{proof} 
    Suppose $N$ is a compact orientable 3-manifold with $\hatpi(N)\cong \hatpi(\X L)$. Equip $L$ with an arbitrary orientation.  By \autoref{thm: hyperbolic link}, there is an oriented hyperbolic link $L'$ such that $N\cong \X{L'}$, and there is a perfect isomorphism $\Phi:\hatpi(\X L)\to \hatpi(\X{L'})$.  Thus, by \autoref{prop: branched cover}, there is an isomorphism $\phi:\hatpi(M_2(L))\to \hatpi(M_2(L'))$ that fits into the commutative diagram:
    \begin{equation*}
        \begin{tikzcd}
            \hatpi(C_2(L)) \arrow[r,"\check{\Phi}"] \arrow[d,two heads,"\widehat{\mathrm{incl}_\ast}"] & \hatpi(C_2(L'))   \arrow[d,two heads,"\widehat{\mathrm{incl}_\ast}"]  \\
            \hatpi(M_2(L)) \arrow[r,"\phi"] & \hatpi(M_2(L'))
        \end{tikzcd}
    \end{equation*}
    According to \autoref{prop: Montesinos two fold cover} the two-fold branched cover $M_2(L)$ is a Seifert fibered space whose orbit space is a $2$-sphere with 3 cone points; and in particular, $M_2(L)$ is not a lens space. By \autoref{thm: Wilton-Zalesskii}, $M_2(L')$ is also a Seifert fibered space. In addition, $C_2(L)$ and $C_2(L')$ are cusped hyperbolic 3-manifolds, so $\check{\Phi}$ is strongly regular by \cite[Theorem 1.4]{Xu25}. As a consequence, $\phi$ is also  strongly regular by \autoref{lem:strong regular criteria}. Therefore, according to \autoref{prop: Seifert strongly regular}, $M_2(L)\cong M_2(L')$.

    Consequently, \autoref{prop: two fold Seifert} implies that  $L'$ is either a Seifert link or a Montesinos link. By assumption, $L'$ is a hyperbolic link, so it must be a Montesinos link. Finally, \autoref{prop: Montesinos characterize} implies that $L'$ is isotopic or mirror image to $L$ (as unoriented links), since   $L$ has 3 tangles. Therefore, $N\cong \X{L'}\cong \X{L}$. 
\end{proof}

\subsection{The exceptional cases}
\begin{lemma}\label{lem: gluing}
    Let $M_1$ and $M_2$ be compact orientable Seifert fibered spaces each containing at least two boundary components. Suppose that both $M_1$ and $M_2$ are achiral.  For $i=1,2$, let $T_i$ be a boundary component of $M_i$, and let $\gamma_i$ denote  a regular fiber of $M_i$ on $T_i$. Suppose $f:T_1\to T_2$ and $f':T_1\to T_2$ are two homeomorphisms such that $\Delta(f(\gamma_1),\gamma_2)=\Delta(f'(\gamma_1),\gamma_2)=n$ and $n\in \{1,2,3,4,6\}$, where $\Delta(\cdot,\cdot)$ denotes the geometric intersection number on the torus $T_2$. Then, $M_1\cup_{f}M_2$ is homeomorphic to $M_1\cup_{f'}M_2$. 
\end{lemma}
\begin{sloppypar}

\begin{proof}
    The conclusion is trivial when one of $M_1$ or $M_2$ is homeomorphic to $T^2\times I$. Thus, we may assume that neither of them is  homeomorphic to $T^2\times I$, and so both of them admit a unique Seifert fibration up to isotopy. 
    
    For $i=1,2$, fix a free $\Z$-basis $\{h_i,e_i\}$ of $H_1(T_i)\cong \Z^2$ such that $h_i$ represents the homology class of $\gamma_i$ (up to a choice of orientation).   The homeomorphisms $f:T_1\to T_2$ and $f':T_1\to T_2$, on the level of homology, can be described by matrices
    $$
    f_\ast\begin{pmatrix}
        h_1 & e_1
    \end{pmatrix}=\begin{pmatrix}
        h_2  & e_2
    \end{pmatrix}\begin{pmatrix}
        a & b\\ c & d
    \end{pmatrix} \quad \text{and}\quad  f_\ast'\begin{pmatrix}
        h_1 & e_1
    \end{pmatrix}=\begin{pmatrix}
        h_2  & e_2
    \end{pmatrix}\begin{pmatrix}
        a' & b'\\ c' & d'
    \end{pmatrix}, 
    $$
    where  $\left(\begin{smallmatrix}
        a & b \\ c & d
    \end{smallmatrix}\right),\left(\begin{smallmatrix}
        a' & b' \\ c' & d'
    \end{smallmatrix}\right)\in\GL_2(\Z)$. 
    Then, $\Delta(f(\gamma_1),\gamma_2)=|c|$ and $\Delta(f'(\gamma_1),\gamma_2)=|c'|$, so $|c|=|c'|=n\in \{1,2,3,4,6\}$.

    We can also identify $\Aut(H_1(T_i))$ with $\GL_2(\Z)$, so that a matrix $P\in \GL_2(\Z)$ acts on $H_1(T_i)$ by 
    $
    \begin{pmatrix}
        h_i & e_i
    \end{pmatrix} \mapsto  \begin{pmatrix}
        h_i & e_i
    \end{pmatrix} P
    $. 
    There is a homomorphism $\rho_i:\mathrm{Homeo}(M_i,T_i)\to \Aut(H_1(T_i))\cong \GL_2(\Z)$ defined in the apparent sense, and we claim that $$
    \mathrm{im}(\rho_i)=U:=\left\{\left.\begin{pmatrix}
        x & z \\ 0 & y 
    \end{pmatrix}\right |\; x,y\in \{\pm1\},\,z\in \Z\right\}.
    $$
    
    On one hand, since $M_i$ admits a unique Seifert fibration up to isotopy,  $\mathrm{Homeo}(M_i)$  preserves the Seifert fibration up to isotopy. Thus, any automorphism in $\mathrm{im}(\rho_i)$ preserves the homology class $h_i$ up to $\pm$-sign. Consequently, $\mathrm{im}(\rho_i)\subseteq U$. On the other hand, \cite[Lemma 8.21]{Xu25A} implies that $\mathrm{im}(\rho)$ contains the subgroup $\left\{\left( \begin{smallmatrix}
        1 & \ast \\ 0  & 1
    \end{smallmatrix}\right )\right\}$, and \cite[Lemma 8.23]{Xu25A} implies that $\mathrm{im}(\rho)$ contains the element $ \left( \begin{smallmatrix}
        -1 & 0 \\ 0  & -1
    \end{smallmatrix}\right ) $. Thus, $\mathrm{im}(\rho)\supseteq U\cap \SL_2(\Z)$. Moreover, $M_i$ is achiral, so $\mathrm{Homeo}(M_i,T_i)$ contains an orientation-reversing element. Thus, $\mathrm{im}(\rho)$ further contains an element with determinant $-1$, and so $\mathrm{im}(\rho)\supseteq U$.

     Note that $M_1\cup_{f}M_2\cong M_1\cup_{f'}M_2$ if there exist $g_1\in \mathrm{Homeo}(M_1,T_1)$ and $g_2\in \mathrm{Homeo}(M_2,T_2)$ such that $f'= g_2|_{T_2}\circ f \circ g_1|_{T_1}$.  
    Since homeomorphisms between tori are uniquely determined up to isotopy by their actions on the first homology, $M_1\cup_f M_2\cong M_1\cup_{f'}M_2$ if there exists $P_1\in U = \mathrm{im}(\rho_1)$ and $P_2\in U=\mathrm{im}(\rho_2)$ such that 
    $$
        \begin{pmatrix}
        a & b\\ c & d
    \end{pmatrix}= P_2 \begin{pmatrix}
        a' & b'\\ c' & d'
    \end{pmatrix} P_1.    $$

   In order to show the existence of such $P_1$ and $P_2$, we claim that there exist $Q_1,Q_2\in U$ such that $$Q_2 \begin{pmatrix}
        a & b\\ c & d
    \end{pmatrix} Q_1=\begin{pmatrix}
        1 & 0\\ n & 1
    \end{pmatrix} .$$
In fact, $|ad-bc|=1$, so $\gcd(a,n)=\gcd(a,c)=1$. For $n\in \{1,2,3,4,6\}$, we must have $a\equiv \pm1\pmod n$. Let $x_2\in \{\pm 1\}$ such that $x_2a \equiv 1\pmod n$,  and let  $y_2=\frac{n}{c} \in \{\pm1\}$. Then, $y_2d\equiv y_2(x_2a)d\equiv y_2x_2(ad-bc)\pmod n$. Let $y_1=y_2x_2(ad-bc)\in \{\pm 1\}$, so $y_1y_2d\equiv 1\pmod n$. Finally, let $z_1= \frac{1-y_1y_2d}{n} \in \Z$ and   $z_2=y_2\frac{1-x_2a}{n} \in \Z$. Then, 
\begin{equation*}
    \begin{pmatrix}
        x_2 & z_2\\ 0 & y_2 
    \end{pmatrix}\begin{pmatrix}
        a & b \\ c & d
    \end{pmatrix}\begin{pmatrix}
        1 & z_1 \\ 0 &  y_1
    \end{pmatrix}=\begin{pmatrix}
        x_2a+z_2c & x_2y_1b+z_2y_1d+z_1(x_2a+z_2c)\\
        y_2c& y_2z_1c+y_2y_1d
    \end{pmatrix}= \begin{pmatrix}
        1 & 0 \\ n & 1 
    \end{pmatrix},
\end{equation*}
where $Q_1=\left(\begin{smallmatrix}
    1 & z_1 \\ 0 & y_1
\end{smallmatrix}\right)\in U$ and $Q_2=\left(\begin{smallmatrix}
    x_2 & z_2 \\ 0 & y_2
\end{smallmatrix}\right)\in U$.  Likewise, there exists $Q_1',Q_2'\in U$ such that $Q_2'\left(\begin{smallmatrix}
    a' & b' \\ c' & d' 
\end{smallmatrix}\right) Q_1'=  \left(\begin{smallmatrix}
    1& 0 \\ n & 1 
\end{smallmatrix}\right)$. Let $P_1=Q_1'Q_1^{-1}\in U$ and $P_2=Q_2^{-1}Q_2'\in U$. Then, $\left(\begin{smallmatrix}
    a & b \\ c & d 
\end{smallmatrix}\right)=P_2\left(\begin{smallmatrix}
    a' & b' \\ c' & d' 
\end{smallmatrix}\right)P_1$, finishing the proof.  
\end{proof}

\end{sloppypar}

\begin{theorem}\label{thm: Montesinos 2}
    Let $L$ be a non-hyperbolic Montesinos link. Then, $\pi_1(\X L)$ is profinitely rigid in $\mathscr{M}$. 
\end{theorem}

\begin{sloppypar}

\begin{proof}
    When $L$ is  a Seifert link, it follows again from   \cite[Corollary 8.3 and Remark 8.4]{Xu25A} that $\pi_1(\X{L})$ is profinitely rigid in $\mathscr{M}$.  Thus, it suffices to consider the four exceptional cases listed in \autoref{prop: Montesinos geometrization classification}~\ref{M.classification.3}.

      In these cases, $\X L$ is a graph manifold with two JSJ-pieces (denoted by $M_1$ and $M_2$) glued along one JSJ-torus, see \autoref{fig: exceptional Montesinos}. Throughout the proof, let $X$ denote the Seifert fibered space $(0,2;\frac{1}{2})$, and let $Y$ denote the Seifert fibered space $(0,3;)=\Sigma_{0,3}\times S^1$.  When $L= \mathbf{M}(\frac{2}{3},-\frac{1}{3},-\frac{1}{3})$ or $L= \mathbf{M}(\frac{1}{2},-\frac{1}{3},-\frac{1}{6})$, both JSJ-pieces of $\X L$ are homeomorphic to $X$; when  $L=\mathbf{M}(\frac{1}{2},-\frac{1}{4},-\frac{1}{4})$, one of the JSJ-pieces of $\X L$ is homeomorphic to $X$, and the other one is homeomorphic to $Y$; when $L=\mathbf{M}(\frac{1}{2},\frac{1}{2},-\frac{1}{2},-\frac{1}{2})$, both JSJ-pieces of $\X L $ are homeomorphic to $Y$.

      \begin{figure}[ht!]
        \centering
        \subfigure[{\scriptsize $\mathbf{M}(\frac{2}{3},-\frac{1}{3},-\frac{1}{3})$}]{\includegraphics[width=2.5cm]{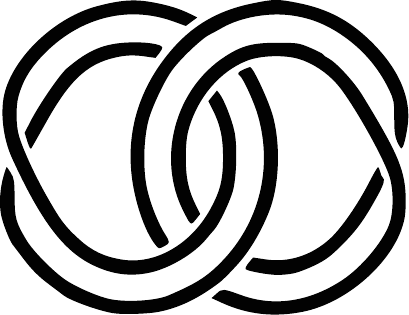}}
        \hspace{1cm}
        \subfigure[{\scriptsize $\mathbf{M}(\frac{1}{2},-\frac{1}{4},-\frac{1}{4})$}]{\includegraphics[width=2.5cm]{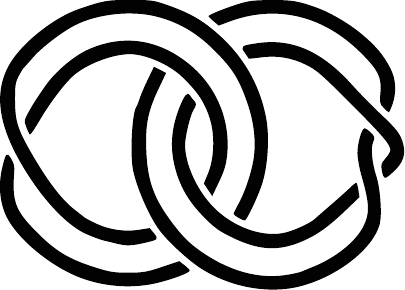}}
        \hspace{1cm}
        \subfigure[{\scriptsize $\mathbf{M}(\frac{1}{2},-\frac{1}{3},-\frac{1}{6})$}]{\includegraphics[width=2.5cm]{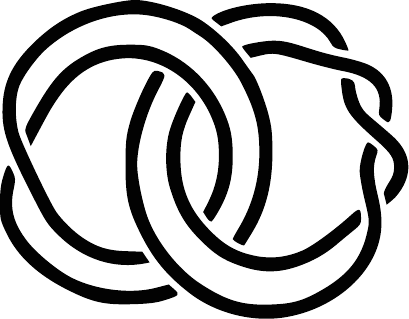}}
        \hspace{0.9cm}
        \subfigure[{\scriptsize $\mathbf{M}(\frac{1}{2},\frac{1}{2},-\frac{1}{2},-\frac{1}{2})$}]{\includegraphics[width=2.9cm]{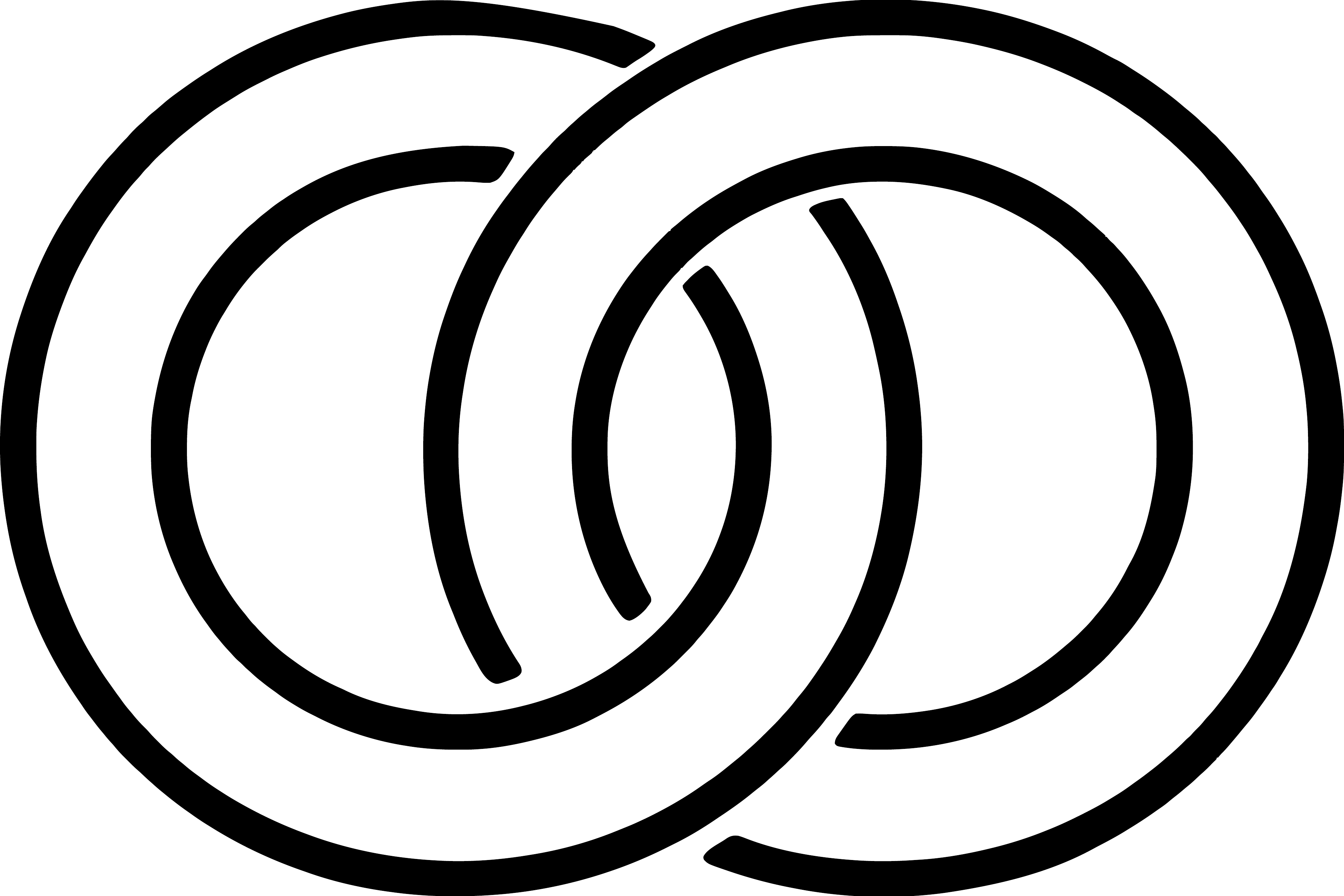}}
        \label{fig: exceptional Montesinos}
        \caption{\cite[Figure 1.4]{Oer84}}
    \end{figure}

    Suppose $N$ is a compact orientable 3-manifold  such that $\hatpi(N)\cong \hatpi(\X L)$. We show that $N$ is homeomorphic to $\X L$. 
    
    \cite[Lemma A.1]{Xu24} and \autoref{thm: Wilton-Zalesskii} imply that $N$ is also a graph manifold with two JSJ-pieces (denoted by $N_1$ and $N_2$) glued along one JSJ-torus, with the property that $\hatpi(M_1)\cong \hatpi(N_1)$ and $\hatpi(M_2)\cong \hatpi(N_2)$. According to \cite[Corollary 8.3]{Xu25A}, this implies that $\pi_1(M_1)\cong \pi_1(N_1)$ and $\pi_1(M_2)\cong \pi_1(N_2)$. Note that any orientable  Seifert fibered space  with fundamental group  isomorphic to $\pi_1(X)$ is homeomorphic to $X$, so $N_i\cong M_i$ if $M_i\cong X$. On the other hand, there are two orientable  Seifert fibered spaces with fundamental group isomorphic to $\pi_1(Y)\cong F_2\times \Z$, namely $Y=\Sigma_{0,3}\times S^1$ and $\Sigma_{1,1}\times S^1$. However, each $M_i$ contains a boundary component of $M$, so according to \cite[Lemma 9.7 (1)]{Xu25}, each $N_i$ also contains a boundary component of $N$. In particular, $\partial N_i$ has at least two components. Therefore, when $M_i\cong Y$, the only remaining possibility is that $N_i \cong Y \cong M_i$. 

    In addition, \cite[Lemma 9.3 (1)]{Xu25A} implies that the geometric intersection number between the regular fibers of $N_1$ and $N_2$ on the JSJ-torus of $N$ is equal to the geometric intersection number between  the regular fibers of $M_1$ and $M_2$ on the JSJ-torus of $S^3-L$, which is $3$ when $L= \mathbf{M}(\frac{2}{3},-\frac{1}{3},-\frac{1}{3})$, and $1$ when $L=\mathbf{M}(\frac{1}{2},-\frac{1}{4},-\frac{1}{4}),\mathbf{M}(\frac{1}{2},-\frac{1}{3},-\frac{1}{6}), \mathbf{M}(\frac{1}{2},\frac{1}{2},-\frac{1}{2},-\frac{1}{2})$. Note that both $X$ and $Y$ are achiral, so \autoref{lem: gluing} implies that $N=N_1\cup_{T^2}N_2$ is homeomorphic to $M_1\cup_{T^2}M_2=\X{L}$.
\end{proof}

\end{sloppypar}

\section{Further criteria for profinite rigidity}\label{motivatingcriteria}

\subsection{Orbifold Dehn surgery}
A more precise version of \autoref{prop: branched cover} can be obtained by profinitely aligning  the so called $\pi$-orbifold groups. 

Let $L=K_1\cups K_n$ be a link in $S^3$. For an integer $r\ge 2$, the {\em $\frac{2\pi}{r}$-orbifold} determined by $L$ is an orbifold with underlying space $S^3$, singular locus $L$, and cone angle $\frac{2\pi}{r}$ on each component $K_i$. We denote the $\frac{2\pi}{r}$-orbifold as $O_r(L)$. Indeed, $O_r(L)$ can be viewed as an orbifold Dehn surgery of $\X{L}$ corresponding to the non-primitive slopes $r\mathbf{m}=(r  m_1,\cdots, r  m_n)$, where the $m_i$'s are the   meridians of $L$ assigned with  arbitrary orientations. Hence, the orbifold fundamental group $\piorb(O_r(L))$ is isomorphic to $\pi_1(\X L)/\langle\!\langle m_1^r,\cdots,m_n^r\rangle\! \rangle$.  

\begin{proposition}\label{prop: orbifold surgery}
     Suppose $L=K_1\cups K_n$ and $L'=K_1'\cups K_n'$ are two component-ordered oriented $n$-component links  in $S^3$,  and  $\Phi: \hatpi(\X{L})\to \hatpi(\X{L'})$  is a perfect isomorphism. Then for each $r\ge 2$, there is an  isomorphism  $\Phi_r:\hatpiorb(O_r(L))\to \hatpiorb(O_r(L'))$   that fits into the following commutative diagram.
    \begin{equation*} 
        \begin{tikzcd}[column sep=large]
            \hatpi(\X{L}) \arrow[d,"\Phi"' ]  \arrow[r, two heads]& \hatpiorb(O_r(L))  \arrow[d,"\Phi_r" ]     \\
            \hatpi(\X{L'})  \arrow[r, two heads] &   \hatpiorb(O_r(L'))
        \end{tikzcd}
    \end{equation*}
\end{proposition}
\begin{proof}
    The Dehn filling alignment theorem (\autoref{thm: dehn filling}) also holds for non-primitive slopes yielding orbifolds, see \cite[Remark 3.6]{Xu24}; and the conclusion follows directly from this generalized version. 
\end{proof}

The following criterion focuses on the $\pi$-orbifold group of a hyperbolic link. 
Let $\mathscr{O}=\{\piorb(O_2(L))\mid L\text{ is a hyperbolic link in }S^3\}$. 

\begin{theorem}\label{cor:20Dehnfilling}
    Suppose $L$ is a hyperbolic link in $S^3$ which is not a two-bridge link, and $\piorb(O_2(L))$ is profinitely rigid in $\mathscr{O}$. Then, $\pi_1(\X L)$ is profinitely rigid in $\mathscr{M}$. 
    \end{theorem}
    \begin{proof} 
Equip $L$ with an arbitrary orientation. 
Suppose $M$ is a compact orientable  3-manifold with $\hatpi(\X L)\cong \hatpi(M)$. According to  \autoref{thm: hyperbolic link}, there is an oriented hyperbolic link $L'\subseteq S^3$ such that $M\cong \X{L'}$ and there  is a perfect isomorphism $\Phi:\hatpi(\X L)\to \hatpi(\X{L'})$. By \autoref{prop: orbifold surgery}, $\hatpiorb(O_2(L))\cong \hatpiorb(O_2(L'))$, so our assumption implies that $\piorb(O_2(L))\cong \piorb(O_2(L'))$. Since a hyperbolic link must be prime and unsplittable, and $L$ is not a two-bridge link,  by \cite[Theorem 3.1]{BoileauZimmermann}, $\X L$ and $\X {L'}$ are diffeomorphic, and so $\pi_1(M)\cong \pi_1(\X{L'})\cong \pi_1(\X{L})$. 
\end{proof}

\subsection{More on cyclic branched covers}
We may also generalize \autoref{prop: branched cover} to $r$-fold cyclic branched covers with $r\ge 3$. To avoid ambiguities in cyclic branched covers, we restrict our attention to knots in this subsection. 

Let $K$ be a knot in $S^3$. 
There is a unique $r$-fold cyclic cover of $\X{K}$, denoted by $C_r(K)$, which is given by the kernel of the unique epimorphism   $\pi_1(\X{K})\twoheadrightarrow H_1(\X{K})\twoheadrightarrow \Z/r\Z$. The $r$-fold cyclic branched cover of $S^3$ along $K$, denoted by $M_r(K)$, is obtained by Dehn filling $C_r(K)$ along the lift of the meridian of $\X{K}$. From another perspective, $M_r(K)$ is the $r$-fold orbifold cover of $O_r(K)$ corresponding to the kernel of $\piorb(O_r(K))\to \piorb(O_r(K))^{\mathrm{ab}}\cong \Z/r\Z$. 

\begin{proposition}\label{prop: r-fold cyclic}
    Suppose $K$ and $K'$ are oriented knots in $S^3$, and $\Phi: \hatpi(\X K)\to \hatpi(\X{K'})$ is a perfect isomorphism. Then, $\hatpi(M_r(K))\cong \hatpi(M_r(K'))$ for any $r\ge 2$. 
\end{proposition}
\begin{proof}
    \autoref{prop: orbifold surgery} gives an isomorphism $\Phi_r:\hatpiorb(O_r(K))\to\hatpiorb( O_r(K'))$. Taking profinite abelianizations gives the following commutative diagram.
    \begin{equation*}
    \begin{tikzcd}
                \hatpiorb(O_r(K)) \arrow[r,"\mathrm{Ab}"] \arrow[d,"\Phi_r"] & \Z/r\Z \arrow[d,"\cong"] \\ 
                \hatpiorb(O_2(L')) \arrow[r,"\mathrm{Ab}"] & \Z/r\Z
    \end{tikzcd}
    \end{equation*}
    The kernels of the abelianizations are exactly the closures of $\pi_1(M_r(K))$ and $\pi_1(M_r(K'))$, which are isomorphic to $\hatpi(M_r(K))$ and $\hatpi(M_r(K'))$. Thus, restricting $\Phi_r$ to the kernels gives an isomorphism $\hatpi(M_r(K))\cong \hatpi(M_r(K'))$. 
\end{proof}

\begin{theorem}\label{cor:branchedcovers}
 Let  $K$ be  a hyperbolic knot in $S^3$. Suppose that for infinitely many integers $r\ge 2$,  $\pi_1(M_r(K))$ is profinitely rigid in $\mathscr{M}$. Then, $\pi_1(\X{K})$ is profinitely rigid in $\mathscr{M}$. 
 \end{theorem}
 
    \begin{proof}
Equip $K$ with an arbitrary orientation. 
   Suppose $M$ is a compact orientable  3-manifold with $\hatpi(\X K)\cong \hatpi(M)$. According to  \autoref{thm: hyperbolic link}, there is an oriented hyperbolic knot $K'\subseteq S^3$ such that $M\cong \X{K'}$ and there exists a perfect isomorphism $\Phi:\hatpi(\X K)\to \hatpi(\X{K'})$. According to  \autoref{prop: r-fold cyclic}, $\hatpi(M_r(K))\cong \hatpi(M_r(K'))$ for all $r\ge 2$. The assumption then implies that $\pi_1(M_r(K))\cong \pi_1(M_r(K'))$ for infinitely many $r\ge 2$. 

   According to \cite[Lemma 4]{Kojima}, $M_r(K)$ is a hyperbolic manifold for $r$ sufficiently large. Thus, Mostow's rigidity theorem implies that $M_r(K)\cong M_r(K')$ for infinitely many $r$. Finally, the main theorem of \cite{Kojima} implies that $K$ and $K'$ are equivalent up to isotopy and mirroring. Thus, $\pi_1(M)\cong \pi_1(\X{K'})\cong \pi_1(\X{K})$. 
    \end{proof}

\bibliographystyle{amsalpha}
\bibliography{main.bib}
\end{document}